\documentclass[11pt]{article}
\usepackage[english]{babel}
\usepackage[utf8]{inputenc}
\usepackage[T1]{fontenc}
\usepackage{hyperref}
\usepackage[a4paper, margin=3cm]{geometry}
\usepackage{dsfont, graphicx, float, stmaryrd, enumerate}
\usepackage{pgf,tikz,pgfplots}

\usepackage[style=trad-abbrv, backend=biber, indexing=cite, sorting=none]{biblatex}
\usepackage{amsmath, amssymb, amsthm, amsfonts}

\newcommand{\R}{\mathbb{R}}
\newcommand{\N}{\mathbb{N}}

\newcommand{\M}{\mathcal{M}}
\newcommand{\T}{\mathcal{T}}

\theoremstyle{plain}
\newtheorem{theorem}{Theorem}
\newtheorem{lemma}{Lemma}
\newtheorem{proposition}{Proposition}
\newtheorem{corollary}{Corollary}
\newtheorem{remark}{Remark}
\newtheorem{definition}{Definition}

\newtheorem{example}{Example}

\theoremstyle{remark}

\title{Universality of the Wigner-Gurau limit for random tensors}
\author{Rémi Bonnin}
\date{}

\begin{document}

\maketitle

\begin{abstract}
In this article, we develop a combinatorial approach for studying moments of the resolvent trace for random tensors proposed by Razvan Gurau. Our work is based on the study of hypergraphs and extends the combinatorial proof of moments convergence for Wigner's theorem. This also opens up paths for research akin to free probability for random tensors.
\\ Specifically, trace invariants form a complete family of tensor invariants and constitute the moments of the resolvent trace. For a random tensor with entries independent, centered, with the right variance and bounded moments, we prove the convergence of the expectation and bound the variance of the balanced single trace invariant. This is the universality of the convergence of the moments of the tensor towards the limiting moments given by the Fuss-Catalan numbers, which are the moments of the law obtained by Gurau in the Gaussian case. This generalizes Wigner's theorem for random tensors. 
\\ Additionally, in the Gaussian case, we show that the limiting distribution of the moments of the $k$-times contracted $p$-order random tensor by a deterministic vector is always the one of a dilated Wigner-Gurau law at order $p-k$. This establishes a connection with the approach of random tensors through study of the matrices given by the $p-2$ contractions of the tensor.
\end{abstract}

\tableofcontents

\section{Introduction}

The study of random tensors has emerged as a generalization of random matrices largely motivated by theoretical physics \cite{gur}. In theoretical physics, random tensors have become indispensable tools for modeling complex systems and phenomena. From quantum field theory to condensed matter physics, random tensor networks provide a powerful framework for understanding emergent behavior and quantum entanglement in multipartite systems \cite{colgur1} \cite{colgur2} \cite{sasa2}. 

More recently, the topic received an important amount of attention linked to computer science \cite{sedd0}. Indeed random tensors find applications in a myriad of questions in computer science, ranging from machine learning to computational biology. Tensor-based methods leverage the high-dimensional structures encoded in data, enabling efficient representation and analysis of complex datasets \cite{sasa1}. The techniques used to study these objects are various and differ according to the point of view adopted. Typically, Jagannath, Lopatto and Miolane proved an interesting analogy of the BBP transition for tensors using techniques of statistical physics \cite{jagannath}, useful because in many applications, prior knowledge on the process that produces the observations leads to a low-rank tensor model (e.g. in \cite{anand, sidi, land}).

Random tensors are more tedious objects than random matrices. Several central notions of the random matrices theory can be generalized with varying degrees of success. For instance, unlike the matrix case, there is no unique notion of eigenvalue for tensors but a few relevant definitions have been proposed \cite{qi}. The most interesting notion of eigenvalue for our perspective is the notion of $z$-eigenvalue. The properties of eigenvalues and eigenvectors of tensors are much more intricate than for matrices, and relatively poorly understood. The number of z-eigenvalues is typically exponentially large in the number of dimensions \cite{cart} \cite{breid}. 
Another central notion for matrices is the one of resolvent trace. In a theoretical physics perspective, Razvan Gurau proposed a generalization of this object for tensors in \cite{gur1} in 2020. The main result of our work is to develop a moments method to compute the asymptotic of the moments of this resolvent trace for a Wigner tensor. It is a combinatorial perspective of a Wigner's Theorem for random tensors.

\subsection{Model and main results} 

\paragraph{Symmetric tensors.} 
For $p \geq 1$, a real-valued $p$-order tensor is a function 
$$ \T : \{ 1,\ldots,N_1 \} \times \ldots \times \{ 1,\ldots,N_p \} \mapsto \mathbb{R}.$$
The integers $N_j, 1\leq j \leq p$ are called the dimensions of the tensor. The entries of a $p$-order tensor $\mathcal{T}$ will be denoted $\mathcal{T}_{i_1\ldots i_p}$. For the following, all the tensors will be of dimension $N_1=\ldots=N_p= N \geq 1$.

\begin{definition}[Symmetric tensor]
We say that a $p$-order tensor $\mathcal{T}$ is \textit{symmetric} if and only if
$$ \forall i_1,\ldots ,i_p, \forall \sigma \in \mathfrak{S}(p), \ \ \mathcal{T}_{i_1\ldots i_p} = \mathcal{T}_{i_{\sigma(1)}\ldots i_{\sigma(p)}}. $$
We will denote $\mathcal{S}_p(N)$ the set of real-valued symmetric tensors.
\end{definition}

\noindent The orthogonal group $\mathbf{O}(N)$ acts on $\mathcal{S}_p(N)$ as follows: for $\T \in \mathcal{S}_p(N)$ and $U \in \mathbf{O}(N)$,
\[ \Big(U \bullet \mathcal{T} \Big)_{i_1,\ldots i_p} := \sum_{j_1,\ldots j_p} \mathcal{T}_{j_1,\ldots j_p} U_{i_1j_1}\ldots U_{i_pj_p}. \]

\paragraph{Combinatorial maps and invariants.} 
We briefly introduce a notion useful for the following. 
A {\em combinatorial map} $b=(\sigma,\tau)$ is the data of two permutations on a set of halfedges: $\sigma$ an arbitrary permutation and $\tau$ an involution with no fixed point. 
It encodes a graph $G(b)=(V(b),E(b))$ with a cyclic order given on the edges around each vertex. 
The vertices $V(b)$ are the cycles of $\sigma$. The edges $E(b)$ are the cycles of $\tau$, that are the pairs of halfedges $(x,\tau(x))$ matched by $\tau$. More generally, $b=(\sigma,\tau)$ is a {\em combinatorial hypermap} if $\tau$ has cycles of arbitrary length, that define the hyperedges. 
Formal definitions are given in Section \ref{sec:back}. 
We say that $b$ is $p$-regular (or $p$-valent) if any vertex belongs to $p$ edges (all cycles of $\sigma$ have length $p$) and $r$-uniform if any hyperedge contains $r$ vertices (all cycles of $\tau$ have length $r$). 
A combinatorial map is then a $2$-uniform combinatorial hypermap. 

\begin{definition}[Trace invariant]\label{def:TI}
Let $\T \in \mathcal{S}_p(N)$ and $b$ be a $p$-regular combinatorial map. 
We denote $(\delta(v)_1,\ldots,\delta(v)_p)$ the sequence of neighboring edges of a vertex $v$. 
The polynomial in the entries of $\T$,
$$\mathrm{Tr}_b(\T):= \sum_{ 1\leq a_1,\ldots,a_{\vert E(b) \vert} \leq N} \prod_{v \in V(b)} \mathcal{T}_{a_{\delta(v)_1} \ldots a_{\delta(v)_p}} ,$$
is called the trace invariant associated to $b$. 
\end{definition}

Why should we focus on trace invariants? 
First, trace invariants form a complete family of the polynomials (in the entries of a tensor) invariant by the action of the orthogonal group, see Proposition \ref{prop:TI}. 
Second, if we define the balanced single trace invariant of size $n$ as the sum over all the trace invariants associated to rooted connected combinatorial maps with $n$ vertices (this set is denoted $\mathcal{B}_n$), that is
$$ I_n(\T):= \sum_{b \in \mathcal{B}_n} \mathrm{Tr}_b(\T) ,$$
then $\frac{1}{N} I_n(\T)$ is the $n$-th moment of the {\em resolvent trace}, see Proposition \ref{prop:res}. 
This resolvent trace is an object introduced by Gurau in \cite{gur1} which generalizes by many aspects the resolvent trace of matrices, see Section \ref{sec:res}. 
In particular in the matrix case, for $\mathcal{M}$ a real symmetric matrix, 
$$\frac{1}{N}I_n(\mathcal{M})=\frac{1}{N}\mathrm{Tr}(\mathcal{M}^n)$$ as the only connected $2$-regular graph is the cycle.

For $\pi$ a partition on the edges of $b$, we denote $b_{\pi}$ the combinatorial hypermap where we merged the edges of a same block of $\pi$ and we denote $\mathrm{Tr}^0_{b_\pi}(\T)$ the trace invariants when summing on distinct indices $a_1,\ldots,a_{\vert \pi \vert}$ (called "injective traces" in the theory of traffics of Male \cite{male0}). Then we have 
$$ \mathrm{Tr}_b(\T) = \sum_{\pi \in \mathcal{P}(E(b))} \mathrm{Tr}^0_{b_\pi}(\T) . $$ 
We also consider the following dual version. 
If $b=(\sigma,\tau)$, we define 
$$b^\dagger =(\tau, \sigma),$$ 
having vertices the cycles of $\tau$ and hyperedges the cycles of $\sigma$. 
Here $H(b^\dagger)=(V(b^\dagger),E(b^\dagger))$ is a hypergraph.
Indeed, if $b$ is a $p$-regular combinatorial map, then $b^{\dagger}$ is a $p$-uniform $2$-regular hypermap, see Definition \ref{def:hyp}. 

\paragraph{Melonics and hypertrees} 
Finally, we say that a hypergraph is a {\em hypertree} if it has no cycle, see Definition \ref{def:cycle}. 
A {\em double hypertree} is a hypergraph where each hyperedge has multiplicity $2$, and where the reduced hypergraph (obtained after forgetting the multiplicity of the hyperedges) is a hypertree.  
\\ If $H(b_\pi^\dagger)$ is a $p$-uniform double hypertree, then $b$ is called a melonic map and $G(b)$ is a {\em melonic graph}. 
The family of $p$-regular melonic graphs is well known in the physics literature. They can be constructed recursively. 
The only one with two vertices is the {\em melon} graph, it is $2$ vertices linked together by $p$ edges. 
Then, a melonic graph with $2n$ vertices is obtained by switching the endpoint of an edge of a melonic graph with $2(n-1)$ vertices with the endpoint of an edge of the melon graph (this operation is called "insertion" of a melon inside an edge). We invite the reader to see Figure \ref{fig:hd} about the duality between double hypertrees and melonics. 

\paragraph{Wigner tensor.} 
Let $\mathcal{T} \in \mathcal{S}_p(N)$ with entries given by random variables that are independent up to symmetries, centered, with finite moments and the right invariant second moment:
\[ \mathbb{E}\left[ \mathcal{T}^2_{i_1,\ldots,i_p}\right] = \frac{1}{(p-1)!}\prod_{a=1}^N c_a(i_1,\ldots,i_p)!, \]
where $c_a(i_1,\ldots,i_p) = \vert \{ 1\leq k \leq p : i_k=a \} \vert$. 
Remark that this variance parameter is equal to $p$ over the number of distinct permutations of the tuple $(i_1,\ldots ,i_p)$. In particular, for the off-diagonal entries, $i.e.$ when $i_1,\ldots,i_p$ are all distinct, the variance of the entry is $1/(p-1)!$. In fact, for our purposes, only the variance of the tensor off-diagonal entries has to be given by 
\[ \mathbb{E}\left[ \mathcal{T}^2_{i_1,\ldots,i_p}\right] = \frac{1}{(p-1)!}, \]
whatever the other ones. 
Indeed, as in the matrix case only the off-diagonal entries variance do matter in our proofs. 

Then, we define a {\em Wigner tensor} of dimension $N$ as the normalized symmetric random tensor
$$ \mathcal{W}_N:= \frac{\T}{N^{\frac{p-1}{2}}}. $$
We are now ready to state our main results.

\begin{theorem}[Moments convergence] \label{thm:1}
    Let $p\geq 2$ and $\mathcal{W}=\left( \mathcal{W}_{N} \right)_{N\geq 1}$ be a sequence of Wigner tensors of order $p$. 
    For all $n\geq 0$, when $N \rightarrow \infty$,
    \[ \mathbb{E} \left[ \frac{1}{N} I_n \left(\mathcal{W}_N\right)\right] = \mathbb{1}_{\text{n even }} F_p\left( \frac{n}{2} \right) + \mathcal{O}\left( \frac{1}{N}\right) ,\]
    where \[ F_p(k) = \frac{1}{pk+1}\binom{pk+1}{k}\] are the Fuss-Catalan numbers.
\end{theorem}

The Fuss-Catalan numbers are the moments of a probability measure $\mu^{(p)}_{\infty}$. 
This universal law $\mu^{(p)}_{\infty}$ is a particular free Bessel law \cite{collins}. 
For $p=2$, we find the Wigner semicircle law, with moments given by the Catalan numbers. So this result is universal and generalizes the Wigner's theorem for tensors. 
The generalization even goes further. 
As in the matrix case, the terms $\frac{1}{N}\mathbb{E}\mathrm{Tr}^0_{b_{\pi}}(\mathcal{W}_N)$ does not vanish if and only if $H(b_{\pi}^{\dagger})$ is a double hypertree. As said before, the dual version of double hypertrees are the melonic. Hence, Theorem \ref{thm:1} relies on the following Lemma.

\begin{lemma}\label{lem:tree}
    Let $p \geq 3$, $b$ be a $p$-regular combinatorial map and $\pi$ a partition of its edges. Then, when $N \rightarrow \infty$,
    $$ \frac{1}{N}\mathbb{E} \mathrm{Tr}^0_{b_{\pi}}(\mathcal{W}_N) \rightarrow \alpha \mathbb{1}_{H(b_{\pi}^{\dagger}) \text{ is a double hypertree}}, $$
    or equivalently, 
    $$ \frac{1}{N}\mathbb{E} \mathrm{Tr}_{b}(\mathcal{W}_N) \rightarrow \alpha \mathbb{1}_{G(b) \text{ is a melonic graph}}, $$
    with $\alpha = (p-1)!^{-V(b)/2} = (p-1)!^{-E(b^\dagger)/2}$
\end{lemma}

We emphasize that we can relax the assumptions made on the entries of the Wigner tensors while keeping Theorem \ref{thm:1} and Lemma \ref{lem:tree} valid. Indeed, the convergence in probability towards the same limit still holds if the entries are $i.i.d.$ having a symmetric law with only \emph{$p$ finite moments} (or $p+1$ in the non-$i.i.d.$ case). The proof is given is Section \ref{sec:pmom}.

The derivation of Theorem \ref{thm:1} from Lemma \ref{lem:tree} is essentially a counting argument. 
In particular, for $p\geq 3$ the rooted planar melonic maps, as well as the rooted plane fully directed hypertrees, as well as the non-crossing partition with blocks of size multiple of $p-1$, as well as $(p-1)$-Dyck paths, are all counted by the Fuss-Catalan numbers. 
See Remark \ref{rem:matrice} for the analogy and the slight difference in the matrix case.

Note that these results only depend on the second moment of the entries of the tensor and not greater ones. 
Moreover, the variance of the moments of the resolvent trace can be bounded.

\begin{theorem}[Variance] \label{thm:2}
    Let $p\geq 2$ and $\mathcal{W}=\left( \mathcal{W}_{N} \right)_{N\geq 1}$ be a sequence of Wigner tensors of order $p$. 
    For all $n\geq 0$, when $N \rightarrow \infty$,
    \[ \mathrm{Var} \left[ \frac{1}{N} I_n \left(\mathcal{W}_N\right)\right] = \mathcal{O}\left( \frac{1}{N^2}\right).\]
\end{theorem}

\noindent Again, the Theorem relies on the following Lemma.

\begin{lemma}\label{lem:var}
Let $b$ and $b'$ be two $p$-regular combinatorial map with $\vert V(b)\vert = \vert V(b') \vert$. Then, when $N \rightarrow \infty$,
    \[ \vert \mathbb{E} \left[ \mathrm{Tr}_{b}(\mathcal{W}_N)\mathrm{Tr}_{b'}(\mathcal{W}_N)\right] - \mathbb{E} \left[ \mathrm{Tr}_{b}(\mathcal{W}_N)\right] \mathbb{E}\left[\mathrm{Tr}_{b'}(\mathcal{W}_N)\right] \vert = \mathcal{O}\left( 1 \right) .\]
\end{lemma}

\noindent Interestingly, it is not true in general that $\mathbb{E} \left[ \mathrm{Tr}_{b}(\mathcal{W}_N)\mathrm{Tr}_{b'}(\mathcal{W}_N)\right]$ and $\mathbb{E} \left[ \mathrm{Tr}_{b}(\mathcal{W}_N)\right] \mathbb{E}\left[\mathrm{Tr}_{b'}(\mathcal{W}_N)\right]$ are always equivalent at large $N$, but it is then at least the case when $\mathbb{E} \left[ \mathrm{Tr}_{b}(\mathcal{W}_N)\right] \mathbb{E}\left[\mathrm{Tr}_{b'}(\mathcal{W}_N)\right]$ is of order $N^2$ (two melonics) or $N$ (one melonic and one of order $1$). 

We may furthermore characterize the limit law $\mu^{(p)}_{\infty}$.

\begin{corollary}[Universal measure]
    The universal measure $\mu^{(p)}_{\infty}$ has odd moments equal to zero and even moments given by the Fuss-Catalan numbers. 
    It has a compact support on 
    \[ \left[ -\sqrt{\frac{p^p}{(p-1)^{p-1}}},\sqrt{\frac{p^p}{(p-1)^{p-1}}} \right] \] 
    and it can be expressed explicitly using hypergeometric functions (see Section \ref{sec:mes}). 
    Its Cauchy-Stieltjes transform, denoted $\mathcal{R}_{\infty}(z):= \sum_{k\geq 0} \frac{F_p(k)}{z^{2k+1}}= \int \frac{d \mu_\infty (\lambda)}{z-\lambda} $, satisfies the identity:
\begin{equation}\label{eq:res}
    z^{p-2} \mathcal{R}_{\infty}(z)^p - z\mathcal{R}_{\infty}(z) +1 = 0.
\end{equation}
\end{corollary}

Equation \eqref{eq:res} generalizes the one known in the matrix case. Moreover, this free Bessel law also appears as the limiting law in the context of a product of $p-1$ Ginibre matrices \cite{collins}. 
It also appears (for the same reason) in the context of Muttalib-Borodin gas, see \cite{forresterwang}. 
Further works may be done to try to understand these links. 
For $p=3$, this law has the following profile.

\begin{figure}[H]
    \centering
    \includegraphics[scale=0.6]{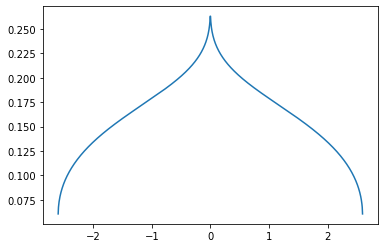}
    \caption{Limit measure for $p=3$.} \label{fig_Wig}
\end{figure}

It is an open question that the resolvent trace of a tensor $\T$ is the Cauchy-Stieltjes transform probability measure $\mu_{\mathcal{T}}$. In this case, we would have $\int \lambda^n d\mu_{\mathcal{T}}(\lambda) =  \frac{I_n(\mathcal{T})}{N}$, then Theorem \ref{thm:1} and Theorem \ref{thm:2} imply $\mu_{\mathcal{W}} \overset{\mathrm{weak}}{\underset{N \rightarrow \infty}{\longrightarrow}} \mu^{(p)}_{\infty}$.

Moreover, the convergence of the moments towards the ones of the universal law is quite robust. 
Indeed, we will show that in the Gaussian case, the moments of a tensor contracted by a vector $k$-times still converge to the moments of the universal law dilated, although we have no longer independent entries after contraction. 

\begin{definition}[Contraction by vectors]\label{def:contraction}
    Let $\mathcal{T} \in \mathcal{S}_p(N)$ and $u^{(1)},\ldots ,u^{(k)}$ be $k$ vectors of $\mathbb{R}^N$ with $k\leq p$. 
    We define the \textit{contraction} $\mathcal{T}\cdot (u^{(1)},\ldots ,u^{(k)})$ as the following $(p-k)$-order tensor
    \[\Big(\mathcal{T}\cdot (u^{(1)},\ldots ,u^{(k)}) \Big)_{i_{k+1},\ldots ,i_p} = \sum_{i_1,\ldots ,i_k}u^{(1)}_{i_1}\ldots u^{(k)}_{i_k} \mathcal{T}_{i_1,\ldots ,i_p} . \]
\end{definition}

\noindent In particular if the contracting vectors are the same, we will denote
\[\mathcal{T}\cdot u^k := \mathcal{T}\cdot \underbrace{(u,\ldots ,u)}_{\text{k times}}. \]

\begin{theorem}[Limit law for the contracted tensor]\label{thm:contr}
    Let $p\geq 3$, $\mathcal{W} \in \mathcal{S}_p(N)$ be a sequence of random tensors belonging to the \textit{Gaussian Orthogonal Tensor Ensemble} (Definition \ref{defW}) and let $u \in \mathbb{S}^{N-1}$ be a sequence of deterministic unit vectors. 
    For $k\leq p-2$, we define $\widetilde{W} := N^{k/2} \mathcal{W} \cdot u^k \in \mathcal{S}_{p-k}(N)$ the normalized contraction of $\mathcal{W}$ by $u^{\otimes k}$. 
    Then, for all $n\geq 0$,
    \[ \frac{1}{N}I_n(\widetilde{W}) \underset{N \longrightarrow \infty}{\rightarrow} \int \lambda^n d \widetilde{\mu}_{\infty}(\lambda), \]
    where 
    \[ \widetilde{\mu}_{\infty}(y) = \sqrt{\binom{p-1}{k}} \mu^{(p-k)}_{\infty}\left(y\sqrt{\binom{p-1}{k}}\right), \]
    with $\mu^{(p-k)}_{\infty}$ the Wigner-Gurau law of order $p-k$, and $\widetilde{\mu}_{\infty}$ supported on 
    \[ \left[-\sqrt{\frac{(p-k)^{p-k}}{\binom{p-1}{k}(p-k-1)^{p-k-1}}}, \sqrt{\frac{(p-k)^{p-k}}{\binom{p-1}{k}(p-k-1)^{p-k-1}}} \right] \]
\end{theorem}

This makes the link with the study of random tensors considering contractions of them. In the particular case $k=p-2$ where the contraction is a matrix, it gives a result of Couillet, Comon, Goulart \cite[Theorem 2]{couillet} as explained in Remark \ref{rem-cou}. It would be interesting to derive this result also without the Gaussian assumption and study the case where the contracting vectors are not all the same, as done in the matrix case by Au and Garza-Vargas \cite{au2023spectralasymptoticscontractedtensor}.

\subsection{Going further} 
We are convinced that the objects and the methods displayed in this paper give the opportunity to study several questions inspired by the ones existing in the field of random matrices.
\begin{enumerate}
    \item This combinatorial approach of the moments opens a very interesting path to develop a free probability theory for random tensors. We are currently working in this perspective.
    \item An interesting extension would be to get a central limit theorem for tensors. We expect that it will follow quite naturally after this work.
    \item It would also be interesting in the future to obtain concentration inequalities for $\mathrm{Tr}_b(\T)$.
    \item Another open question is if there exists a link between the limiting measure and a notion of eigenvalues for tensor, in particular the $z$-eigenvalues. In the matrix case, the Wigner semi-circle law gives the asymptotic distribution of the eigenvalues. The $z$-eigenvalues are still the critical points of the energy involved in the definition of the resolvent trace but their relation with the limiting spectral measure is not understood. We are not going to address this point in this work 
\end{enumerate}

\subsection{Organization of the paper} 
The second part provides the background about tensors and invariants. In the third part, we present the proofs of our results.

\subsection{Acknowledgements}
I would like to warmly thank Charles Bordenave and Djalil Chafaï for guiding me through this topic and for all the fruitful discussions with them. I am also grateful to Benoit Collins for his hospitality in Kyoto at the early stage of this work and to Camille Male for his hospitality in Bordeaux and for their informed advice. Last but not least, I thank Alexis Imbert for all the great discussions about this project.


\section{Tensors and invariants} \label{sec:back}

\subsection{Generalities about tensors}

We first recall some definitions about random tensors, some of them have already appeared in the introduction. 
Fix $p\geq 1$ and $N\geq 1$. 
A tensor $\T \in (\R^N)^{\otimes p}$ is said symmetric if for all indices $i_1,\ldots,i_p$ and all permutation $\sigma \in \mathfrak{S}(p)$, 
$$\T_{i_1,\ldots,i_p}=\T_{i_{\sigma(1)},\ldots,i_{\sigma(p)}}.$$
We denote $\mathcal{S}_p(N)$ the set of real symmetric tensors of order $p$ and dimension $N$.

The contraction of the $r$-th leg of a $p$-order tensor $\T$ with the $r'$-th leg of a $p'$-order tensor $\T'$ is the $(p+p'-2)$-order tensor given by
$$ (\T \cdot \T')_{i_1,\ldots,i_{p+p'-2}} = \sum_{j=1}^N \T_{i_1,\ldots,i_{r-1},j,i_{r},\ldots,i_{p-1}} \T'_{i_p,\ldots,i_{p+r'-2},j,i_{p+r'-1},\ldots,i_{p+p'-2}} .$$
When $\T \in \mathcal{S}_p(N)$ and we contract the $k$ first legs by tensors of order $1$ ($i.e.$ vectors), we retrieve the contraction by vectors as in Definition \ref{def:contraction}. 
When we contract each one of the $p$ legs by an orthogonal matrix $U \in \mathbf{O}(N)$, it is the following multilinear transform.

\begin{definition}[Multilinear transform]\label{mt}
    Let $\T \in \mathcal{S}_p(N)$ and $U \in \mathbf{O}(N)$. We define
    \[ U \bullet \Big(\mathcal{T} \Big)_{i_1,\ldots i_p} := \sum_{j_1,\ldots j_p} \mathcal{T}_{j_1,\ldots j_p} U_{i_1j_1}\ldots U_{i_pj_p} . \]
\end{definition}

\noindent We also have a scalar product on $\mathcal{S}_p(N)$ given by the contraction of the $p$ legs of a tensor $\T$ with the $p$ ones of $\T'$,
\[ \langle \mathcal{T}, \mathcal{T'} \rangle := \sum_{j_1,\ldots j_p} \mathcal{T}_{j_1,\ldots j_p} \mathcal{T'}_{j_1,\ldots j_p} , \]
and the Euclidian (or Frobenius) norm is 
$$\| \mathcal{T} \|_F:= \sqrt{\langle \mathcal{T}, \mathcal{T} \rangle}.$$

\begin{remark}
A $z$-eigenpair for $\T$ is a pair $(\lambda,u)$ such that $u$ is a unit vector and 
$$ \T \cdot u^{p-1} = \lambda u. $$
Remark also that if we denote $u^{(1)}\otimes\ldots \otimes u^{(p)}$ the tensor with coordinates $(u^{(1)}\otimes\ldots \otimes u^{(p)})_{i_1,\ldots ,i_p} = u^{(1)}_{i_1}\ldots u^{(p)}_{i_p}$, it is immediate to check that 
\[\langle \mathcal{T},u^{(1)}\otimes\ldots \otimes u^{(p)} \rangle = \langle \mathcal{T}\cdot (u^{(1)},\ldots ,u^{(p-1)}), u^{(p)} \rangle = \mathcal{T}\cdot (u^{(1)},\ldots ,u^{(p)}). \]
Hence the largest eigenvalue of $\T$ is the maximum for $u \in \mathbb{S}^{N-1}$ of $\langle \mathcal{T},u^{\otimes p} \rangle$. 
\end{remark}

Finally, let $\mathcal{T} \in \mathcal{S}_p(N)$ and let us consider a polynomial in the entries of $\mathcal{T}$:
\[ P(\mathcal{T}) = \sum_{k=0}^{D} \sum_{a = (a^i_v)_{\genfrac{}{}{0pt}{}{1\leq v \leq k}{1 \leq i \leq p}}} P_a \prod_{v=1}^k \mathcal{T}_{a^1_v \ldots  a^p_v}\]
where $P_a \in \R$ and the sum runs over $a \in \{1,\ldots ,N \}^{k \times p}$.

\begin{definition}[Invariant]\label{def:inv}
    Such a polynomial is said invariant if it is invariant under the action of the orthogonal group by multilinear transform (Definition \ref{mt}), $i.e.$ 
    \[ \forall U \in \mathbf{O}(N), P(\mathcal{T}) = P(U \bullet \mathcal{T} ) . \]
\end{definition}

\subsection{Combinatorial maps and trace invariants}

We first introduce a tool convenient to encode a complete family of invariants.

\begin{definition}[Combinatorial map]
    A combinatorial map $b$ is a triple $(Q,\sigma,\tau)$ where 
    \begin{itemize}
        \item[$\bullet$] $Q= \{ q_1,\ldots ,q_{r} \}$ is a finite set of halfedges.
        \item[$\bullet$] $\sigma$ is a permutation on $Q$. The cycles of $\sigma$ are called the vertices $V(b)$ of $b$.
        \item[$\bullet$] $\tau$ is an involution on $Q$ with no fixed point. The cycles $ \{ (q, \tau(q)) \}$ of $\tau$ are called the edges $E(b)$ of $b$.
    \end{itemize} 
A combinatorial map is said rooted if one of its halfedges is marked. The set $Q$ is often omitted and considered to be canonically $\{1,\ldots,2r \}$. 
\end{definition}

It is a graph with a cyclic order of the edges around each vertex. 
We denote $G(b)=(V(b), E(b))$ the graph when forgetting the order around the vertices. 
Note that $G(b)$ may have self-loops and several edges between two vertices, it is formally a multigraph. 
Two combinatorial maps are equivalent if they differ only by a relabelling of the halfedges, that is $b \sim b'$ if there exists $\theta \in \mathfrak{S}(Q)$ such that $\sigma'= \theta \circ \sigma \circ \theta^{-1}$ and $\tau'= \theta \circ \tau \circ \theta^{-1}$. 
If $b, b'$ are rooted in $r$ and $r'$ we require moreover that $\theta(r)=r'$. 
We say that $b$ is unlabelled if we consider the conjugacy class of $b$ under $\sim$.

If $\tau$ has cycles of arbitrary length, we say that $b=(\sigma, \tau)$ is a combinatorial hypermap, where hyperedges are the cycles of $\tau$. Hence a combinatorial map is simply a $2$-uniform combinatorial hypermap, meaning that $\tau$ has only cycles of length $2$.

A combinatorial map $b=(\sigma,\tau)$ is said $p$-regular (or $p$-valent) if $\sigma$ has only cycles of length $p$, that is each vertex belongs to $p$ edges (including self-loops, counted two times). 
If $b$ is $p$-regular and has $n$ vertices, then it has $np/2$ edges. 
We denote $\mathbf{B}_n^{(p)}$, or simply $\mathbf{B}_n$ when there is no ambiguity, the set of $p$-regular combinatorial maps with $n$ unlabelled vertices. 
Remark that if $p$ and $n$ are odd, then $\mathbf{B}_n^{(p)} =\emptyset$ as the number of halfedges $np$ is odd so there is no possible matching. Then, we define $\mathbf{B}^{(p)}:= \sqcup_{n\geq0} \mathbf{B}_n^{(p)}$. 
Similarly, we denote $\mathcal{B}_n^{(p)}$ the set $p$-regular combinatorial map with $n$ unlabelled vertices that are rooted and connected, in the sense that $G(b)=(V(b),E(b))$ is a connected graph. Again, we also denote
$$\mathcal{B}^{(p)}:= \sqcup_{n\geq0} \mathcal{B}_n^{(p)}.$$

\begin{example}
For $p=3$, there are five rooted unlabelled combinatorial maps with $2$ vertices, given in Figure \ref{fig:cm}. 
\begin{figure}[H]
    \centering
    \includegraphics[scale=0.4]{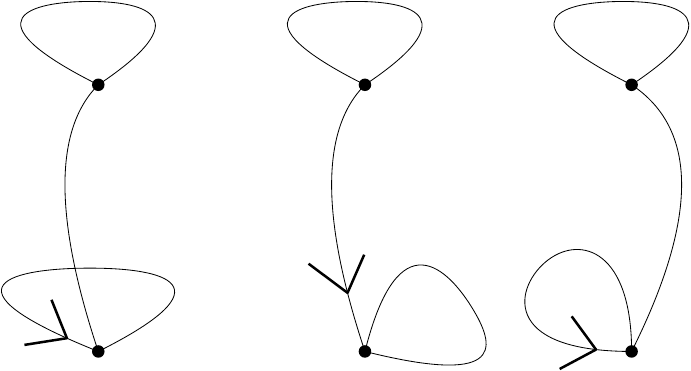} \ \ \ \ \ \
    \includegraphics[scale=0.4]{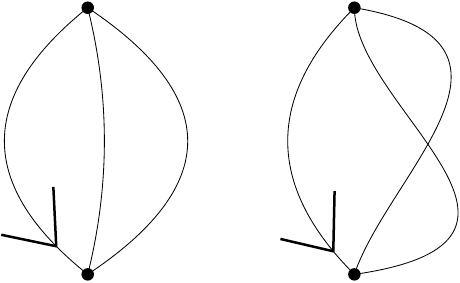}
    \caption{Combinatorial maps for $p=3$ and $n=2$. The dart marks a halfedge (the root). For instance, with the root given by $1$ and $\sigma=(1,2,3)(4,5,6)$ fixed, the first map is given by $\tau_1=(1,3)(2,4)(5,6)$, the second by $\tau_2=(1,4)(2,3)(5,6)$ and the third by $\tau_3=(1,2)(3,4)(5,6)$. The graph $G(b)$ associated to the two maps on the right-hand side, $2$ vertices linked by $p$ edges, is called a melon. The family of all melonic graphs is obtained recursively by inserting a melon inside the edge of a melonic.}\label{fig:cm}
\end{figure}
\end{example}

If $\T \in \mathcal{S}_p(N)$, we can associate to a $p$-regular combinatorial map $b$ the polynomial in the entries of $\T$ given in Definition \ref{def:TI}, 
$$\mathrm{Tr}_b(\T):= \sum_{ 1\leq a_1,\ldots,a_{\vert E(b) \vert} \leq N} \prod_{v \in V(b)} \mathcal{T}_{a_{\delta(v)_1} \ldots a_{\delta(v)_p}} ,$$
where $(\delta(v)_1,\ldots,\delta(v)_p)$ is the sequence of neighboring edges in $b$ of a vertex $v$. 
It is called the trace invariant associated to $b$. 
For instance, the trace invariant associated to a melon map is equal to the square of the Frobenius norm $\| \mathcal{T} \|_F^2 =\sum_{i_1,\ldots,i_p} \T_{i_1,\ldots,i_p}^2 $. 

\begin{proposition}\label{prop:TI}
    The family of trace invariants form a complete family of invariant polynomials. In other words, if $P(\T)$ is an invariant polynomial (in the sense of Definition \ref{def:inv}), then there exists a family of real numbers $\{ P_b, b \in \mathbf{B}^{(p)} \}$ such that 
    $$ P(\T) = \sum_{b \in \mathbf{B}^{(p)}} P_b \mathrm{Tr_b}(\T) .$$
\end{proposition}

This is a classical result. A simple proof can be found in [\cite{kun24}, Theorem 3.2.] or [\cite{bonnin2025characterizationgaussiantensorensembles}, Section 3.1.]. Only the complex case was treated in the book of Gurau \cite{gur}.

\begin{definition}[Balanced single trace invariant]\label{def:bsti}
    For $\T \in \mathcal{S}_p(N)$ and $n\geq 1$, we define the balanced single trace invariant of degree $n$ as the sum over $\mathcal{B}_n^{(p)}$, the rooted connected combinatorial maps with $n$ unlabelled vertices, 
    $$ I_n(\mathcal{T}) := \sum_{b \in \mathcal{B}_n} \mathrm{Tr}_b(\mathcal{T}). $$
By convention we set $I_0(\mathcal{T})=N$.
\end{definition}

One important point in order to prove later the convergence of the balanced single trace invariant of a Wigner tensor is to distinguish how many indices are distinct when considering a trace invariant where the sum is over indices $a_1,\ldots, a_{E(b)}$. 
For $b \in \mathcal{B}_n$, denote $\mathcal{P}(E(b))$ the set of partitions of the edges and for $\pi \in \mathcal{P}(E(b))$, denote $b_{\pi}$ the map $b$ where the edges in a same block of $\pi$ are merged. 
If $E(b)=\{1,\ldots,np/2\}$, we denote $\mathcal{P}(np/2)=\mathcal{P}(E(b))$ and we denote $\vert \pi \vert$ the number of blocks of $\pi$. 
Then,
$$ I_n(\mathcal{T}) = \sum_{b \in \mathcal{B}_n} \sum_{\pi \in \mathcal{P}(np/2)} \mathrm{Tr^0}_{b_{\pi}}(\mathcal{T}), $$
where $$ \mathrm{Tr^0}_{b_{\pi}}(\mathcal{T}):= \sum_{ \genfrac{}{}{0pt}{}{1\leq a_1,\ldots,a_{\vert \pi \vert} \leq N}{\text{distinct}} } \prod_{v \in V(b)} \mathcal{T}_{a_{\delta(v)_1} \ldots a_{\delta(v)_p}}. $$
Note that if $b$ was a combinatorial map, $b_\pi$ is now a combinatorial hypermap where we merge the cycles of $\tau$ belonging to a same block of $\pi$. 
In this sense, if $b$ was a $p$-regular combinatorial map, that is a $p$-regular and $2$-uniform combinatorial hypermap, then $b_\pi$ is only a $p$-regular combinatorial hypermap.

\begin{example}
According to Figure \ref{fig:cm}, we have for $p=3$,
$$ I_2(\T) = 3 \sum_{a,b,c} \T_{aab} \T_{bcc} + 2 \sum_{a,b,c} \T_{abc}^2. $$
We give the different partitions of the edges of the first combinatorial map of the Figure \ref{fig:cm}.
\begin{figure}[H]\label{fig:pi}
    \centering
    \includegraphics[scale=0.5]{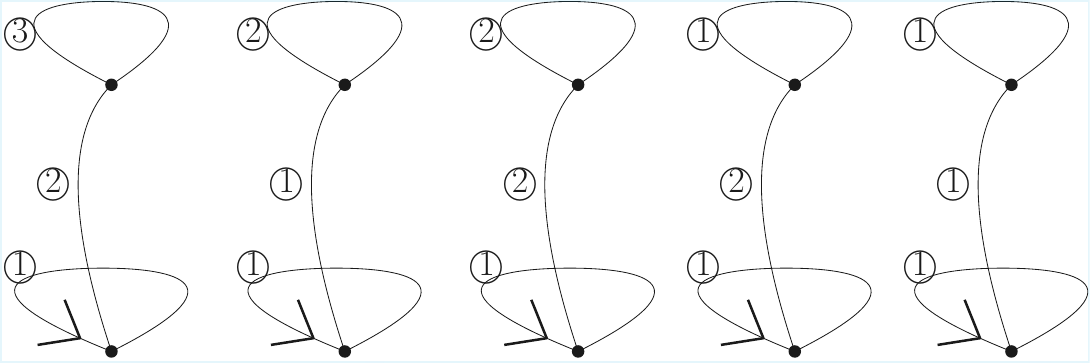} 
\caption{Possible partitions of the edges. We have $ \mathrm{Tr}_b(\T) = \sum_{ \genfrac{}{}{0pt}{}{a,b,c}{\text{distinct}} } \T_{aab} \T_{bcc} + 2 \sum_{a \neq b} \T_{aaa} \T_{abb} + \sum_{a \neq b} \T_{aab}^2 + \sum_{a} \T_{aaa}^2$.}
\end{figure}
\end{example}

\subsection{A resolvent for tensors}\label{sec:res}

A motivation to introduce the balanced single trace invariant is that they appear in the expansion of the resolvent trace, an object introduced by Gurau in \cite{gur1}. 
We are going to recall briefly the main known characteristics of this trace of resolvent, the interested reader may refer to the work of Gurau for more details.

\begin{definition}[Resolvent trace]
    For $\T \in \mathcal{S}_p(N)$, we define the balanced resolvent trace of $\mathcal{T}$, for $z \in i\mathbb{R}$,
    \begin{equation} \notag
    \mathcal{R}_{\T}(z) := \frac{z^{-1}}{\Xi(z)}\int \frac{\|\phi\|^2}{N} \exp \left(-\left( \frac{1}{2}\|\phi \|^2 - \frac{1}{z}\frac{\mathcal{T} \cdot \phi^p}{p}\right)\right) [d\phi ]
\end{equation}
where 
\begin{itemize}
    \item[$\bullet$] $[d\phi ] := (2 \pi)^{-N/2} \prod_{i=1}^N d\phi_i$
    \item[$\bullet$] $\mathcal{T} \cdot \phi^p := \sum_{1 \leq i_1,\ldots ,i_p \leq N} \mathcal{T}_{i_1,\ldots ,i_p} \phi_{i_1}\ldots \phi_{i_p}$
    \item[$\bullet$] $\Xi_{\T}(z) := \int \exp \left(-\left( \frac{1}{2}\|\phi \|^2 - \frac{1}{z}\frac{\mathcal{T} \cdot \phi^p}{p}\right)\right) [d\phi ]$
\end{itemize} 
This resolvent trace is well defined on two cones containing $i\mathbb{R}^+$ and $i\mathbb{R}^-$ respectively, and admits an analytic continuation on $\mathbb{C} \setminus \mathbb{R}$. 
It is given for $z=|z|e^{i\alpha}$ in $\mathbb{H}^+ := \{x+iy : x \in \mathbb{R}, y>0 \}$, by
\[ \mathcal{R}_{\T}^+(z) := \frac{z^{-1}}{\Xi_{\T}(z)}\int e^{\frac{i}{p}N(\alpha-\frac{\pi}{2})} \frac{\|\phi\|^2 }{N} \exp \left(-\left( \frac{1}{2}\|\phi \|^2 e^{\frac{2i}{p}(\alpha-\frac{\pi}{2})} - \frac{1}{z}\frac{\mathcal{T} \cdot \left(\phi e^{\frac{i}{p}(\alpha-\frac{\pi}{2})}\right)^p}{p}\right)\right) [d\phi ], \]
and similarly on $\mathbb{H}^-$. 
We refer to the appendix A of the Gurau's article \cite{gur1} for the proof of this fact.
\end{definition}

In the matrix case we find the usual resolvent trace with this formula. We give a proof of this fact in the Appendix \ref{appendixA} at the end of this work. For the following of this Section, we fix $\T \in \mathcal{S}_p(N)$ and we denote $\mathcal{R}(z)$ for $\mathcal{R}_{\T}(z)$.

\begin{proposition} \label{res}
    We have the relation on $\mathbb{C} \setminus \mathbb{R}$,
    \[ \mathcal{R}(z) = z^{-1} - \frac{p}{N} \frac{d}{dz}\left( \mathrm{ln } \Xi(z) \right). \]
\end{proposition}

\begin{proof}
    Integrating by parts, we have for $1 \leq i \leq N$:
\[ \Xi(z) = - \int \phi_i \left( -\phi_i + \frac{1}{z} \sum_{i_2,\ldots ,i_p} \mathcal{T}_{i, i_2\ldots i_p}\phi_{i_2}\ldots \phi_{i_p} \right) \exp \left(-\left( \frac{1}{2}\|\phi \|^2 - \frac{1}{z}\frac{\mathcal{T} \cdot \phi^p}{p}\right)\right) [d\phi ] . \]
Hence summing on $i$, 
    \[ N \Xi(z) = \int \left( \|\phi\|^2 - \frac{1}{z} \mathcal{T} \cdot \phi^p \right) \exp \left(-\left( \frac{1}{2}\|\phi \|^2 - \frac{1}{z}\frac{\mathcal{T} \cdot \phi^p}{p}\right)\right) [d\phi ] , \]
and then,
\begin{align*}
    \mathcal{R}(z) &= \frac{z^{-1}}{\Xi(z)}\int \frac{\|\phi\|^2}{N} \exp \left(-\left( \frac{1}{2}\|\phi \|^2 - \frac{1}{z}\frac{\mathcal{T} \cdot \phi^p}{p}\right)\right) [d\phi ] \\
    &= z^{-1} - \frac{p}{N} \frac{1}{\Xi(z)} \int \frac{-1}{z^2} \frac{\mathcal{T} \cdot \phi^p}{p}  \exp \left(-\left( \frac{1}{2}\|\phi \|^2 - \frac{1}{z}\frac{\mathcal{T} \cdot \phi^p}{p}\right)\right) [d\phi ] \\
    &= z^{-1} - \frac{p}{N} \frac{1}{\Xi(z)} \frac{d}{dz}\left( \Xi(z) \right) .
\end{align*}
This gives the result.
\end{proof}

This relation gives the possibility to compute the expansion of the resolvent trace on $\mathbb{C} \setminus \mathbb{R}$.

\begin{proposition}[Expansion of the resolvent trace]\label{prop:res}
For $z \in \mathbb{C} \setminus \mathbb{R}$, we have 
    \[ \frac{d^{n+1} \mathcal{R}(z)}{dz^{n+1}} \underset{|z| \rightarrow \infty}{\sim} \frac{I_n(\mathcal{T})}{N}. \]
    Equivalently, the resolvent trace has a formal expansion around $z=\infty$,
    \[ \mathcal{R}(z) = \sum_{n\geq 0} \frac{1}{z^{n+1}} \frac{I_n(\mathcal{T})}{N} .\]
\end{proposition}

\begin{proof}[Sketch of proof]
    We know by Proposition \ref{res} that
    \[ \mathcal{R}(z) = z^{-1} - \frac{p}{N} \frac{d}{dz}\left( \mathrm{ln} \Xi{Z}(z) \right). \]
    We have the expansion around $z=\infty$,
    \[ \Xi(z) := \sum_{n\geq 0} \frac{1}{n!}\frac{1}{z^n} \int \exp \left(- \frac{1}{2}\|\phi \|^2 \right) \left(\frac{\mathcal{T} \cdot \phi^p}{p}\right)^n [d\phi ] . \]
    By Wick Theorem, performing the Gaussian integral consists of choosing which legs of each one of the $n$ copies of $\T$ are matched. 
    It is then choosing the matching $\tau$ on the halfedges of the combinatorial map where each cycle of $\sigma$ is a copy of $\T$. 
    This standard fact is often used in physics literature (see Section 3.2. in \cite{gur2}) and it simplify the expansion into 
    \[ \Xi(z) := \sum_{n\geq 0} \frac{1}{z^n} \sum_{b \in \mathbf{B}_n} \mathrm{Tr}_b(\mathcal{T}) \]
    where we recall that $\mathbf{B}_n$ is the set of (possibly disconnected) combinatorial maps with $n$ unlabelled vertices. 
    It remains to use some basic analytic combinatorics. 
    In particular, with the relation between a generating series counting connected objects and its exponential one counting the disconnected, we get 
    \[ \mathrm{ln} \Xi(z) := \sum_{n\geq 0} \frac{1}{z^{n}} \sum_{ \genfrac{}{}{0pt}{}{b \in \mathbf{B}_n}{\text{connected}} } \mathrm{Tr}_b(\mathcal{T}). \]
    Then, derivation with respect to $z$ marks a vertex (and change the sign as $d(1/z^n)/dz=-n/z^{n+1}$) and the factor $p$ marks a halfedge. Hence,
    \[ \mathcal{R}(z) = z^{-1} + \frac{1}{N} \sum_{n\geq 1} \frac{1}{z^{n+1}} \sum_{b \in \mathcal{B}_n } \mathrm{Tr}_b(\mathcal{T}), \]
    with $\mathcal{B}_n$ the set of rooted connected combinatorial maps with $n$ unlabelled vertices. This finally gives the result,
    \[ \mathcal{R}(z) = \sum_{n\geq 0} \frac{1}{z^{n+1}} \frac{I_n(\mathcal{T})}{N} .\]
\end{proof}

\begin{remark}
    In the matrix case where $p=2$, we have for all $n$, $|\mathcal{B}_n|=1$ because the unique connected $2$-regular graph with $n$ vertices is a cycle. 
    The trace invariant associated to the cycle is equal to $\mathrm{Tr}\left(\mathcal{M}^n\right)$, so $I_n(\mathcal{T}) = \mathrm{Tr}\left(\mathcal{M}^n\right) $. 
    Hence this formal expansion gives in the matrix case 
    \[ \mathcal{R}(z) = \sum_{n\geq 0} \frac{1}{z^{n+1}} \frac{\mathrm{Tr}\left(\mathcal{M}^n\right)}{N}\]
    and we retrieve the classical matrix resolvent trace expansion for $|z| > \|\mathcal{M}\|$:
    \[ \frac{1}{N} \mathrm{Tr}\left( \left(z\mathcal{I}-\mathcal{M}\right)^{-1}\right) = \frac{1}{N} \frac{1}{z}\mathrm{Tr}\left(  (\mathcal{I}-\frac{\mathcal{M}}{z})^{-1}\right) = \sum_{n\geq 0} \frac{1}{z^{n+1}} \frac{\mathrm{Tr}\left(\mathcal{M}^n\right)}{N} \]
\end{remark}

\subsection{Gaussian Orthogonal Tensor Ensemble}

Gurau studied the resolvent trace in the case of a Gaussian tensor. More precisely let us introduce the Gaussian Orthogonal Tensor Ensemble (\textit{GOTE}) that generalizes the matrix Gaussian Orthogonal Ensemble.

\begin{definition}[\textit{GOTE}] \label{defW} 
We said that a symmetric tensor $\mathcal{W} \in \mathcal{S}_p(N)$ belongs to the \textit{GOTE} if, as a tensor-valued random variable, it has a density with respect to the natural Lebesgue measure on $\mathcal{S}_p(N)$ proportional to
\[ f(\mathcal{W}) \propto  
\exp{\left(-\frac{N^{p-1}}{2 p} \| \mathcal{W} \| _F^2 \right)} \]
\end{definition}

As in the matrix case the law of such a tensor is invariant with respect to an orthogonal change of basis because the density $f$ only depends on the Euclidian norm, and
\[\forall U \in \mathbf{O}(N), \ \ \| U \bullet \mathcal{W} \|_F = \|\mathcal{W}\|_F. \]
Moreover, taking into account the symmetry, Definition \ref{defW} implies that the entries of a tensor of the \textit{GOTE} are Gaussian, centered, with variance $\sigma_{i_1,\ldots,i_p}$ depending only of the type of the tuple $(i_1,\ldots,i_p)$, and given by the following Lemma.

\begin{lemma}[Invariant second moment] \label{Pi}
    If $\mathcal{W}$ belongs to the Gaussian Orthogonal Tensor Ensemble, then $\mathcal{W}_{i_1,\ldots,i_p} \sim \mathcal{N}\left(0,\left(\frac{\sigma_{i_1,\ldots,i_p}}{N^{\frac{p-1}{2}}}\right)^2\right)$ where 
    \[ \sigma_{i_1,\ldots,i_p}^2 = \frac{p}{ |\mathcal{P}_{i_1,\ldots,i_p} | } = \frac{1}{(p-1)!}\prod_{a=1}^N \vert \{1\leq k\leq p : i_k=a \} \vert !, \]
    with $\mathcal{P}_{i_1,\ldots,i_p}$ the set of distinct permutations of $( i_1,\ldots ,i_p )$, i.e. the $\theta \in \mathfrak{S}(p)$ such that $( i_1,\ldots ,i_p )=( i_{\theta(1)},\ldots ,i_{\theta(p)} )$. 
\end{lemma}

\begin{proof}
    We write $D:=\{(i_1\leq\ldots \leq i_p) \in \{ 1,\ldots, N\} ^p \text{ distinct up to any permutation} \}$.
    The lemma only relies on the simple fact that 
    \[\| \mathcal{W} \| _F^2 = \sum_{(i_1,\ldots ,i_p) \in D} |\mathcal{P}_{i_1,\ldots,i_p} | \times \mathcal{W}_{i_1,\ldots ,i_p}^2, \]
    hence 
    $$\frac{N^{p-1}}{2 p} \| \mathcal{W} \| _F^2 = \frac{1}{2} \sum_{(i_1,\ldots ,i_p) \in D} \frac{N^{p-1}|\mathcal{P}_{i_1,\ldots,i_p} |}{p} \mathcal{W}_{i_1,\ldots ,i_p}^2 .$$
    We get the result. Note finally that by the orbit-stabilizer Theorem,
    $$ |\mathcal{P}_{i_1,\ldots,i_p} | = \frac{p!}{\prod_{a=1}^N c_a(i_1,\ldots,i_p) !} ,$$
    where $c_a(i_1,\ldots,i_p)= \vert \{1\leq k\leq p : i_k=a \} \vert$.
\end{proof}

\begin{example}
    For instance for $p=3$, we have 
    \[ \| \mathcal{W} \| _F^2 = \sum_a \mathcal{W}_{aaa}^2 + 3\sum_{a\neq b} \mathcal{W}_{aab}^2 + 6\sum_{a<b<c} \mathcal{W}_{abc}^2\]
    and hence,
    \[ \sigma_{ijk}^2 = \frac{1}{2} + \frac{1}{2} \delta_{ij} +\frac{1}{2} \delta_{ik}  + \frac{1}{2} \delta_{jk} +  \delta_{ij} \delta_{ik} \ \ \in \{3, 1, \frac{1}{2} \}\]
    where $\delta_{ij}=1$ if $i=j$ and 0 otherwise.
\end{example}

Using a saddle point method, Gurau proved that the resolvent trace of the \textit{GOTE} is asymptotically given by the generating series of the Fuss-Catalan numbers.

\begin{proposition}[\cite{gur1}]
    Let $(\mathcal{W}_N)_{n\geq 1}$ be a sequence of random tensors belonging to the \textit{GOTE}. 
    There exists $\omega_c >0$ such that for $\vert z \vert < \omega_c$, when $N \rightarrow \infty$,
    $$ \mathcal{R}_{\mathcal{W}_N} (z) \rightarrow \frac{1}{z} T_p \left(\frac{1}{z^2} \right), $$
    where $T_p(z):= \sum_{k\geq 0} F_p(k) z^k$.
\end{proposition}

We are going to derive a moments method to prove that the (even) moments of the resolvent trace always converge to the Fuss-Catalan numbers for a Wigner tensor, without the Gaussian assumption.


\section{Convergence of the moments}

\subsection{Hypergraphs} \label{sec:hyp}

We introduce here the notion of hypergraph, useful for all the following of our work. 
\begin{definition}[Hypergraph]\label{def:hyp}
    A simple hypergraph $H=(V,E)$ is a set of vertices $V$ and a set of hyperedges $E$, that are multiset of vertices. As a vertex $v$ may appear multiple times in a hyperedge $e$, we denote $l_v(e)$ the multiplicity of $v$ in $e$ (equal to $0$ if $v$ not in $e$). 
    \par A hypergraph is more generally the data of $H=(V,E,m)$ where $H^*= (V,E)$ is a simple hypergraph and $m: E \rightarrow \N^*$ is a function associating a multiplicity to each hyperdge. The simple hypergraph $H^*$ is called the reduced hypergraph of $H$, it is the hypergraph after forgetting the multiplicity of the hyperedges.
\end{definition}

\noindent The number of vertices, counted with multiplicity, in a hyperedge $e$ is called the order of $e$, that is 
$$\vert e\vert=\sum_{v\in V} l_v(e).$$
The number of hyperedges a vertex $v$ belongs to is called the degree of $v$, that is 
$$ d(v)=\sum_{e\in E} l_v(e). $$
A hypergraph is \underline{$p$-uniform} if $\vert e \vert = p$ for all $e \in E$, and it is \underline{$r$-regular} if $d(v)=r$ for all $v \in V$.

\begin{definition}\label{def:cycle}
    A cycle of length $k\geq 1$ in a hypergraph is a sequence $(v_1,e_1,v_2,\ldots,e_{k},v_{k+1})$ such that
    \begin{itemize}
        \item for all $i$, $v_i \in V$, $e_i \in E$ and $v_1=v_{k+1}$,
        \item for all $i$, $l_{v_i}(e_i)\geq 1$, $l_{v_{i+1}}(e_i)\geq 1$ and if $v_i=v_{i+1}$, $l_{v_i}(e_i) \geq 2$,
        \item the $e_i$ are distinct.
    \end{itemize}
A simple hypergraph with no cycle is called a hypertree.
\end{definition}

Remark that a hypergraph has a cycle of length $1$ if and only if it has a vertex of multiplicity at least $2$ in a hyperedge. Remark also that a cycle in the hypergraph is a cycle in the bipartite graph obtained after adding a new vertex $w_e$ of type $2$ for each hyperedge $e$ and linking $v \in V$ of type $1$ with all the vertices $w_e$ of type $2$ such that $v\in e$.

We say that a hypergraph $H=(V,E,m)$ is a \underline{double hypertree} if $m(e)=2$ for all $e \in E$ and $H^*=(V,E)$ is a hypertree. In other words all edges have multiplicity $2$ and the reduced hypergraph is a hypertree.

\begin{lemma}[Hypergraph Euler's formula]\label{lem:euler}
    If $H=(V,E)$ is a connected $p$-uniform simple hypergraph, then
    $$ 1 - \vert V \vert + (p-1) \vert E \vert \geq 0, $$
    with equality if and only if $H$ is a hypertree.
\end{lemma}

\begin{proof}
Let $H=(V,E)$ be a connected $p$-uniform simple hypergraph. Consider $G$ the bipartite (multi)graph associated to $H$ with vertices $V(G)=V \sqcup \{w_e , e \in E \}$ and edges $E(G)$ given by the pairs  $(v, w_e)$ such that $v \in e$, with the same multiplicity as $v$ in $e$. This multigraph is a tree if and only if $H$ is a hypertree. The Euler formula applied to this connected graph gives
\begin{align*}
    0 \leq 1 - \vert V(G) \vert + \vert E(G) \vert &= 1 - (\vert V \vert + \vert E \vert)  + p \vert E \vert \\
    &= 1 - \vert V \vert + (p-1) \vert E \vert.
\end{align*} 
\end{proof}

Let $b=(\sigma,\tau)$ be a $p$-regular combinatorial map, $i.e.$ a $2$-uniform and $p$-regular combinatorial (hyper)map. 
Then $b^\dagger=(\tau,\sigma)$ is a $p$-uniform and $2$-regular combinatorial hypermap that is 
$$ H(b^\dagger)=(V(b^\dagger),E(b^\dagger))$$
is a $p$-uniform and $2$-regular hypergraph. 
Now, if $\pi$ is a partition of the edges of $b$ ($i.e$ of the vertices of $b^\dagger$), then $b_\pi$ is a $p$-regular combinatorial hypermap and $b_\pi^\dagger$ is a $p$-uniform combinatorial hypermap (where each vertex belongs to an even number of hyperedges, not necessarily exactly two).

\subsection{Proof of Lemma \ref{lem:tree}}

Let $p\geq 3$ and $\mathcal{T} \in \mathcal{S}_p(N)$ with entries given by random variables that are independent up to symmetries, centered, with finite moments and variance of the off-diagonal entries given by
\[ \mathbb{E}\left[ \mathcal{T}^2_{i_1,\ldots,i_p}\right] = \frac{1}{(p-1)!}. \]

\begin{proof}[Proof of Lemma \ref{lem:tree}]
Denote $\gamma := \frac{p-1}{2}$. 
Let $b \in \mathcal{B}^{(p)}$ and denote $n= \vert V(b) \vert =\vert E(b^\dagger) \vert$ the number of vertices of $b$, we compute
\begin{align*}
    \frac{1}{N}\mathbb{E} \left[ \mathrm{Tr}_b \left(\frac{\T}{N^\gamma}\right) \right] &= \frac{1}{N} \mathbb{E} \left[ \sum_{\pi \in \mathcal{P}(np/2)} \mathrm{Tr^0}_{b_{\pi}}\left(\frac{\T}{N^\gamma}\right) \right] \\
    &= \frac{1}{N} \mathbb{E}\left[ \sum_{\pi \in \mathcal{P}(np/2)} \sum_{ \genfrac{}{}{0pt}{}{1\leq a_1,\ldots,a_{\vert \pi \vert} \leq N}{\text{distinct}} } \prod_{v \in V(b)} \frac{\mathcal{T}_{a_{\delta(v)_1} \ldots a_{\delta(v)_p}} }{N^\gamma} \right] \\
    &= \frac{1}{N} \mathbb{E}\left[ \sum_{\pi \in \mathcal{P}(np/2)} \sum_{ \genfrac{}{}{0pt}{}{1\leq a_1,\ldots,a_{\vert \pi \vert} \leq N}{\text{distinct}} } \prod_{e=(v_1,\ldots,v_p) \in E(b_\pi^\dagger)} \frac{\mathcal{T}_{a_{v_1} \ldots a_{v_p}} }{N^\gamma} \right] \\
    &= N^{-(1 + \gamma n ) } \sum_{\pi \in \mathcal{P}(np/2)} \sum_{ \genfrac{}{}{0pt}{}{1\leq a_1,\ldots,a_{\vert \pi \vert} \leq N}{\text{distinct}} } \mathbb{E}\left[ \prod_{e=(v_1,\ldots,v_p) \in E(b_\pi^\dagger)} \mathcal{T}_{a_{v_1} \ldots a_{v_p}} \right]
\end{align*}
Recall that $H(b_\pi^\dagger)=(V(b_\pi^\dagger),E(b_\pi^\dagger))$ is a $p$-uniform hypergraph, denote $H_\pi^*=(V_\pi^*,E_\pi^*)$ its reduced simple hypergraph and for $e\in E_\pi^*$, $m(e)$ its multiplicity in $H(b_\pi^\dagger)$. 
Then, we have
$$ \mathbb{E}\left[ \prod_{e=(v_1,\ldots,v_p) \in E(b_\pi^\dagger)} \mathcal{T}_{a_{v_1} \ldots a_{v_p}} \right] = \prod_{e=(v_1,\ldots,v_p) \in E_\pi^*} \mathbb{E}\left[ \mathcal{T}_{a_{v_1} \ldots a_{v_p}}^{m(e)} \right]. $$
As the entries are centered, this product is equal to $0$ as soon as there exists $e \in E(b_\pi^\dagger)$ such that $m(e)=1$. 
This means that we must consider only the combinatorial hypermaps $b_\pi$ such that 
$$ \forall e \in E(b_\pi^\dagger), m(e) \geq 2, $$
or equivalently $\forall v \in V(b_\pi), m(v) \geq 2$. 
We deduce that
\begin{equation}\label{eq:1}
    n = \sum_{e\in E_\pi^*} m(e) \geq 2 \vert E_\pi^* \vert,
\end{equation}  
with equality if and only if $m(e)=2$ for all $e$. In particular, $n$ must be even for the equality case.

Recall also that $\pi$ is a partition of the edges of $b$, or equivalently of the vertices of $b^\dagger$. 
Hence, by Lemma \ref{lem:euler},
\begin{equation}\label{eq:2}
    \vert \pi \vert = \vert V_\pi^* \vert \leq 1 + (p-1) \vert E_\pi^* \vert , 
\end{equation} 
with equality if and only if $H_\pi^*$ is a hypertree. 
In particular, in this case, if $e=(v_1,\ldots,v_p)$ is a hyperedge, then $v_1,\ldots,v_p$ are distinct. 
Hence, if $H_\pi^*$ is a hypertree, 
\begin{align*}
    \sum_{ \genfrac{}{}{0pt}{}{1\leq a_1,\ldots,a_{\vert \pi \vert} \leq N}{\text{distinct}} } \prod_{e=(v_1,\ldots,v_p) \in E_\pi^*} \mathbb{E}\left[ \mathcal{T}_{a_{v_1} \ldots a_{v_p}}^{m(e)} \right]
    &= N(N-1)\ldots (N-\vert \pi \vert +1) \prod_{e \in E_\pi^*} \mathbb{E}\left[ \mathcal{T}_{1 \ldots p}^{m(e)} \right] \\
    &= N^{\vert \pi \vert} \prod_{e \in E_\pi^*} \mathbb{E}\left[ \mathcal{T}_{1 \ldots p}^{m(e)} \right] + \mathcal{O}\left( N^{\vert \pi \vert-1}\right) \\
    &= N^{1 + (p-1) \vert E_\pi^* \vert} \prod_{e \in E_\pi^*} \mathbb{E}\left[ \mathcal{T}_{1 \ldots p}^{m(e)} \right] + \mathcal{O}\left( N^{(p-1) \vert E_\pi^* \vert}\right).
\end{align*}
When $H_\pi^*$ is not necessarily a hypertree, we have in all generality that 
$$ \sum_{ \genfrac{}{}{0pt}{}{1\leq a_1,\ldots,a_{\vert \pi \vert} \leq N}{\text{distinct}} } \prod_{e=(v_1,\ldots,v_p) \in E_\pi^*} \mathbb{E}\left[ \mathcal{T}_{a_{v_1} \ldots a_{v_p}}^{m(e)} \right] = \mathcal{O}\left(N^{\vert \pi \vert}\right), $$
because all the moments of the entries of $\T$ are finite. 
Hence we have that 
$$ N^{-(1 + \gamma n ) } \sum_{ \genfrac{}{}{0pt}{}{1\leq a_1,\ldots,a_{\vert \pi \vert} \leq N}{\text{distinct}} } \mathbb{E}\left[ \prod_{e=(v_1,\ldots,v_p) \in E(b_\pi^\dagger)} \mathcal{T}_{a_{v_1} \ldots a_{v_p}} \right] = \mathcal{O}\left(N^{\vert \pi \vert - 1 - \frac{p-1}{2}n} \right). $$
Using Equation \eqref{eq:1} and Equation \eqref{eq:2} gives 
$$ \vert \pi \vert - 1 - \frac{p-1}{2}n \leq \vert \pi \vert - 1 - (p-1)\vert E_\pi^* \vert \leq 0, $$
with equality between left hand side and right hand side if and only if $b_\pi^\dagger$ is a double hypertree (all the hyperedges have multiplicity exactly $2$ and the reduced hypergraph is a hypertree). 
Finally, we proved that 
\begin{equation}\label{eq:dh}
    \frac{1}{N}\mathbb{E} \left[ \mathrm{Tr}^0_{b_\pi} \left(\frac{\T}{N^\gamma}\right) \right]
    = \mathbb{1}_{H(b^\dagger_\pi) \text{ is a double hypertree }} ( \underbrace{\mathbb{E}\left[ \mathcal{T}_{1 \ldots p}^2 \right]}_{=1/(p-1)!} )^{\vert E(b^\dagger) \vert/2} + \mathcal{O}\left( \frac{1}{N} \right).
\end{equation} 
 
\paragraph{Melonics.} 
Let $p\geq 3$ and $b^\dagger_\pi$ be a $p$-uniform combinatorial hypermap such that $H(b^\dagger_\pi)$ is a double hypertree. 
How many $p$-uniform and $2$-regular hypermap may be partitioned into $b^\dagger_\pi$? 
Only one! 
Consider the problem recursively, $H_\pi^*$ is a hypertree so there exists a hyperedge $e=(v_1,\ldots,v_p)$ with $v_2,\ldots,v_p$ being leaves. These leaves by definition belong only to one hyperedge (of multiplicity $2$) in $H(b^\dagger_\pi)$. 
Hence the two copies of $v_1$ must be matched by a compatible $2$-regular hypermap. We can erase this hyperedge and repeat recursively.

Reciprocally, let $b^\dagger=(\tau,\sigma)$ be a $p$-uniform and $2$-regular hypermap that may be partitioned into $b^\dagger_\pi$ such that $H(b^\dagger_\pi)$ is a double hypertree. 
It cannot be partitioned into another hypermap whose associated hypergraph is a double hypertree for the same reason, $H(b^\dagger_\pi)$ must contain a hyperedge with $p-1$ leaves of the reduced hypertree. As $p \geq 3$, two hyperedges sharing $p-1 \geq 2$ in a double hypertree must be the same so this enforce the two copies of the last vertex in the hyperedge to be in the same block of $\pi$ to avoid a cycle. 
Recursively, this gives a unique possible $\pi$. 
We call unfolded double hypertree the hypergraph $H(b^\dagger)$. 
It is a $p$-uniform and $2$-regular simple hypergraph, see Figure \ref{fig:hd}.

We call melonic the dual under $\dagger$ of an unfolded double hypertree. 
It is a $p$-regular graph. 
Equation \ref{eq:dh} gives
\begin{equation}\label{eq:me}
    \frac{1}{N}\mathbb{E} \left[ \mathrm{Tr}_b \left(\frac{\T}{N^\gamma}\right) \right] = \mathbb{1}_{G(b) \text{ is a melonic graph }} \frac{1}{(p-1)!^{\vert V(b) \vert/2}} + \mathcal{O}\left( \frac{1}{N} \right).
\end{equation} 
The result is proved. 
We call {\em melonic map} a combinatorial map such that $G(b)$ is a melonic graph. 
In regard of what we say just before, if $b$ is a melonic map, there is a unique partition $\pi$ such that $b_{\pi}$ is "well partitioned", in the sense that $\frac{1}{N}\mathbb{E} \left[ \mathrm{Tr}^0_{b_\pi} \left(\frac{\T}{N^\gamma}\right) \right]$ does not vanish. 
And no other combinatorial maps may be partitioned into $b_{\pi}$ (it is true for $p\geq 3$, see Remark \ref{rem:matrice} for the slight difference in the matrix case).
\end{proof}

\begin{figure}[H]
    \center
\begin{tikzpicture}[scale=0.55]
    \draw [ double distance = 4pt ] (0,0) -- (3,-4) ;
    \draw [ double distance = 4pt ] (0,0) -- (-3,-4) ;
    \draw [ double distance = 4pt ] (-3,-4) -- (3,-4) ;
    \draw [ double distance = 4pt ] (-3,-4) -- (0,-8) ;
    \draw [ double distance = 4pt ] (0,-8) -- (-6,-8) ;
    \draw [ double distance = 4pt ] (-3,-4) -- (-6,-8) ;
    \filldraw[black] (0,0) circle (5pt);
    \filldraw[black] (-3,-4) circle (5pt);
    \filldraw[black] (-6,-8) circle (5pt);
    \filldraw[black] (3,-4) circle (5pt);
    \filldraw[black] (0,-8) circle (5pt);
    \draw (0,0) node[above right] {$a_1$};
    \draw (3,-4) node[above right] {$a_2$};
    \draw (-3,-4) node[above left] {$a_3=a_6$};
    \draw (0,-8) node[above right] {$a_4$};
    \draw (-6,-8) node[above left] {$a_5$};
    \draw (-5,-1) node[above left] {$H(b^\dagger_\pi)$};
    \draw (10,0) .. controls (10,-1) and (9,-2) .. (8,-3);
    \draw (10,0) .. controls (10,-1) and (11,-2) .. (12,-3);
    \draw (8,-3) .. controls (9,-2) and (11,-2) .. (12,-3);
    \draw (8,-3) .. controls (8,-4) and (7,-5) .. (6,-6);
    \draw (8,-3) .. controls (8,-4) and (9,-5) .. (10,-6);
    \draw (6,-6) .. controls (7,-5) and (9,-5) .. (10,-6);
    \draw (8,-9) .. controls (8,-8) and (7,-7) .. (6,-6);
    \draw (8,-9) .. controls (8,-8) and (9,-7) .. (10,-6);
    \draw (6,-6) .. controls (7,-7) and (9,-7) .. (10,-6);
    \draw (10,0) .. controls (11,-1) and (11.5,-2) .. (12,-3);
    \draw (8,-9) .. controls (10,-8) and (14,-3) .. (12,-3);
    \draw (8,-9) .. controls (15,-6) and (16,-2) .. (10,0);
    \filldraw[black] (10,0) circle (5pt);
    \filldraw[black] (6,-6) circle (5pt);
    \filldraw[black] (8,-3) circle (5pt);
    \filldraw[black] (12,-3) circle (5pt);
    \filldraw[black] (10,-6) circle (5pt);
    \filldraw[black] (8,-9) circle (5pt);
    \draw (10,0) node[above right] {$a_1$};
    \draw (12,-3) node[above right] {$a_2$};
    \draw (8,-3) node[above left] {$a_3$};
    \draw (10,-6) node[above right] {$a_4$};
    \draw (6,-6) node[above left] {$a_5$};
    \draw (8,-9) node[above left] {$a_6$};
    \draw (16,-1) node[above left] {$H(b^\dagger)$};
    \draw (-5.5,-12) .. controls (-4.5,-9) and (1.5,-9) .. (2.5,-12);
    \draw (-5.5,-12) .. controls (-4.5,-11) and (1.5,-11) .. (2.5,-12);
    \draw (-5.5,-12) .. controls (-5.5,-13) and (-4,-14) .. (-3,-14);
    \draw (-3,-14) .. controls (-2,-13) and (-1,-13) .. (0,-14);
    \draw (-3,-14) .. controls (-2,-14.5) and (-1,-14.5) .. (0,-14);
    \draw (2.5,-12) .. controls (2.5,-13) and (1,-14) .. (0,-14);
    \filldraw[black] (-5.5,-12) circle (5pt);
    \filldraw[black] (2.5,-12) circle (5pt);
    \filldraw[black] (0,-14) circle (5pt);
    \filldraw[black] (-3,-14) circle (5pt);
    \draw (-1.5,-10) node {$a_1$};
    \draw (-1.5,-11.5) node {$a_2$};
    \draw (-5.5,-13.5) node {$a_3=a_6$};
    \draw (2.5,-13.5) node {$a_3=a_6$};
    \draw (-1.5,-13) node {$a_4$};
    \draw (-1.5,-14.5) node {$a_5$};
    \draw (-5,-16) node[above left] {$H(b_\pi)$};
    \draw (-5.5+13,-12) .. controls (-4.5+13,-9) and (1.5+13,-9) .. (2.5+13,-12);
    \draw (-5.5+13,-12) .. controls (-4.5+13,-11) and (1.5+13,-11) .. (2.5+13,-12);
    \draw (-5.5+13,-12) .. controls (-5.5+13,-13) and (-4+13,-14) .. (-3+13,-14);
    \draw (-3+13,-14) .. controls (-2+13,-13) and (-1+13,-13) .. (0+13,-14);
    \draw (-3+13,-14) .. controls (-2+13,-14.5) and (-1+13,-14.5) .. (0+13,-14);
    \draw (2.5+13,-12) .. controls (2.5+13,-13) and (1+13,-14) .. (0+13,-14);
    \filldraw[black] (-5.5+13,-12) circle (5pt);
    \filldraw[black] (2.5+13,-12) circle (5pt);
    \filldraw[black] (0+13,-14) circle (5pt);
    \filldraw[black] (-3+13,-14) circle (5pt);
    \draw (-1.5+13,-10) node {$a_1$};
    \draw (-1.5+13,-11.5) node {$a_2$};
    \draw (-5+13,-13.5) node {$a_3$};
    \draw (2+13,-13.5) node {$a_6$};
    \draw (-1.5+13,-13) node {$a_4$};
    \draw (-1.5+13,-14.5) node {$a_5$};
    \draw (16,-16) node[above left] {$G(b)$};
\end{tikzpicture}
    \caption{A double hypertree (above left), unfolded double hypertree (above right), melonic well partitioned (below left) and melonic graph (below right).}\label{fig:hd}
\end{figure}

\subsection{Proof of Theorem \ref{thm:1}}

Fix $p\geq 3$. 
We remind that
$$ \frac{1}{N}\mathbb{E} \left[ I_n \left(\frac{\T}{N^\gamma}\right) \right] = \sum_{b \in \mathcal{B}_n^{(p)}} \frac{1}{N}\mathbb{E} \left[ \mathrm{Tr}_b \left(\frac{\T}{N^\gamma}\right) \right]. $$

\begin{proof}[Proof of Theorem \ref{thm:1}]
It remains, essentially, to pass from Lemma~\ref{lem:tree} to Theorem~\ref{thm:1} to perform a counting argument.

We aim to count rooted melonic maps $b$. It is more convenient to first count those that are planar. We therefore proceed in two steps. First, we count the number of melonic maps canonically associated with a given planar melonic map, obtained by untwisting the edges. Second, we count rooted planar melonic maps.

The first step is straightforward. Passing from a rooted melonic map to its rooted planar structure yields a factor
$$ (p-1)!^{|V(b)|/2} = (p-1)!^{|E(b^\dagger)|/2}.$$
Indeed, let $b=(\sigma,\tau)$ be a rooted melonic map. Recall that $b$ is unlabelled. Assume that $1$ is the root and that $\tau(1)=j$. Denote by $(1,v_1,\ldots,v_{p-1})$ and $(j,w_1,\ldots,w_{p-1})$ the cycles of $1$ and $j$ under $\sigma$, respectively. The equivalence class of $b$ contains all maps $b'=(\sigma',\tau)$ such that the cycle of $j$ under $\sigma'$ is of the form
$$ (j,w_{\theta(1)},\ldots,w_{\theta(p-1)}), $$
for some $\theta\in\mathfrak{S}(p-1)$. There are exactly $(p-1)!$ such maps. One may then repeat the same argument recursively for the cycles of $\tau(v_1),\ldots,\tau(v_{p-1})$, considered as new roots. The melonic (or hypertree) structure ensures that these choices can be made independently. Hence, the equivalence class of $b$ has cardinality $(p-1)!^{\#\sigma/2} = (p-1)!^{|V(b)|/2}$. These classes form a partition of the set of rooted melonic maps. Exactly one representative in each class is planar: namely, the one corresponding to the permutation $\theta$ such that, if the branch of $v_i$ contains $w_k$, then $\theta(k)=i$, for all $i$ and all pairs of vertices of $b$.

Note that the same counting can be performed in the dual formulation. Observe that if $b$ is an unlabelled rooted combinatorial map, then its dual $b^\dagger$ is formally a rooted fully directed hypergraph, in which hyperedges are ordered tuples rather than multisets. The cyclic order of edges around each vertex induces a cyclic order of vertices within each hyperedge. A \emph{fully directed double hypertree} is then defined as an fully directed hypergraph that becomes a double hypertree once the orientations of the hyperedges are forgotten. If $H$ is a rooted fully directed plane simple hypertree, then there are $(p-1)!^{E(H)/2}$ rooted fully directed double hypertrees in the equivalence class of $H$. Indeed, the structure is fixed by the first occurrence of each oriented hyperedge inherited from the root, and there are then $(p-1)!$ possible orderings for the second occurrence of each hyperedge.

The second step is classical. Let $t_n=t_n^{(p)}$ denote the number of rooted planar melonic maps with $2n$ vertices, and let
$$ T_p(z):=\sum_{n\geq 0} t_n z^n $$
be the associated generating series. It satisfies
\begin{equation}\label{eq:gefu}
T_p(z) = 1 + z\, T_p(z)^p,
\end{equation}
since a rooted planar melonic map is either the empty graph, or a single melon in which a melonic graph is inserted into each of the $p$ edges. The same recursive description applies to $p$-uniform rooted fully directed plane hypertrees: such a hypertree is either empty, or consists of a single hyperedge with $p$ rooted hypertrees attached to its vertices. Consequently, the number of rooted fully directed plane hypertrees with $n$ vertices is also given by $t_n$.

Equation \eqref{eq:gefu} is well known to characterize the generating function of the Fuss--Catalan numbers, so that
$$ t_n^{(p)} = F_p(n). $$

From the perspective of random matrix theory, where double trees are counted by Dyck paths, it is instructive to describe an explicit bijection between $p$-uniform rooted fully directed plane hypertrees and $p-1$-Dyck paths. For $n\geq 0$, a $(p-1)$-Dyck path is a lattice path in $\mathbb{Z}^2$ from $(0,0)$ to $(np,0)$, staying weakly above the horizontal axis, and consisting of steps in
$$\{(1,1),(1,-(p-1))\}.$$
Let $H=(V,E)$ be a rooted fully directed plane hypertree. For $v,w\in V$, we write $v\sim w$ if the two vertices appear consecutively in an oriented hyperedge. The distance between the root $v_0$ and a vertex $w$ is defined as the minimal length of a sequence $v_1,\ldots,v_k=w$ such that $v_i\sim v_{i+1}$ for all $i$. We associate to $H$ the sequence of distances from the root obtained by performing a depth-first search of the vertices, following the orientation of the hyperedges and visiting only previously unvisited vertices. Along this traversal, the distance increases by one at each step, except when the exploration of a hyperedge is completed, in which case the distance decreases by $p-1$. The resulting sequence therefore defines a $(p-1)$-Dyck path. It is straightforward to verify that this construction is bijective, by reconstructing the hypertree step by step starting from the end of the $(p-1)$-Dyck path.

Finally, the number of $(p-1)$-Dyck paths is well known (see \cite{cameron}) and given by
$$\frac{1}{(p-1)n+1}\binom{np}{n} = \frac{1}{pn+1}\binom{np+1}{n} = F_p(n).$$
This description also makes it easy to obtain a bijection with non-crossing partitions of $n(p-1)$ elements into blocks of size divisible by $p-1$, which are likewise counted by Fuss--Catalan numbers and are again natural objects from the perspective of random matrix theory.
\end{proof}

\begin{remark}[Matrix case]\label{rem:matrice}
    In the case $p=2$, all the proof of Lemma \ref{lem:tree} is correct until Equation \eqref{eq:dh}, that is 
    $$ \frac{1}{N}\mathbb{E} \left[ \mathrm{Tr}^0_{b_\pi} \left(\frac{\T}{N^\gamma}\right) \right]
    = \mathbb{1}_{G(b^\dagger_\pi) \text{ is a double tree }} + \mathcal{O}\left( \frac{1}{N} \right).$$
    Then it is still true that the number of rooted (double) plane trees is equal to the number of Dyck path and is equal to the Catalan numbers. 
    Hence, it is also still true that $\frac{1}{N}\mathbb{E} \left[ \mathrm{Tr}^0_{b_\pi} \left(\frac{\T}{N^\gamma}\right) \right]$ does not vanish if and only if $b_\pi$ is a melonic map well partitioned, that is in this context 
    $$ \frac{1}{N}\mathbb{E} \left[ \mathrm{Tr}^0_{b_\pi} \left(\frac{\T}{N^\gamma}\right) \right] = \mathbb{1}_{\pi \text{ is a non-crossing partition }} + \mathcal{O}\left( \frac{1}{N} \right).$$
    The picture is only slightly different according to the fact that there is only one $2$-regular rooted connected map with $n$ unlabelled vertices, it is the cycle, but now it can be well partitioned by several partitions which are the non-crossing partitions. 
    Indeed, the argument that two hyperedges sharing $p-1$ vertices must me the same in a double hypertree is no longer holding for $p=2$ (indeed two neighbor edges of a double tree surely share a vertex). 
    The result may now be written as
\begin{align*}
    \frac{1}{N}\mathbb{E} \left[ I_{2n} \left(\frac{\mathcal{M}}{\sqrt{N}}\right) \right] = \frac{1}{N}\mathbb{E} \left[ \mathrm{Tr}_{\mathrm{cycle}} \left(\frac{\M}{\sqrt{N}}\right) \right]
    &= \sum_{\pi \in \mathcal{P}(n)} \left(\mathbb{1}_{\pi \text{ is a non-crossing partition }} + \mathcal{O}\left( \frac{1}{N} \right)\right) \\ 
    &= F_2(n) + \mathcal{O}\left( \frac{1}{N} \right).
\end{align*} 
\end{remark}

\subsection{p finite moments}\label{sec:pmom}

Let $p\geq 3$ and $\mathcal{T} \in \mathcal{S}_p(N)$. 
Our goal in this section is to understand when does the previous result hold if we delete the assumption that all the moments of the entries are finite. 
We consider the following assumptions:
\begin{itemize}
    \item[$(i)$] the entries $\mathcal{T}_{i_1,\ldots,i_p}$ are centered with variance $\frac{1}{(p-1)!}$,
    \item[$(ii)$] the entries $\mathcal{T}_{i_1,\ldots,i_p}$ are $i.i.d.$ having the same law as a random variable $X$, 
    \item[$(iii)$] the entries $\mathcal{T}_{i_1,\ldots,i_p}$ have $p$ finite moments,
    \item[$(iv)$] the law of the entries is symmetric.
\end{itemize}
We discuss about the necessity of these assumptions in Remark \ref{rem:nec}. 
We will prove that we still have convergence in probability towards the Fuss-Catalan numbers if $\mathcal{T}$ satisfies $(i)-(iv)$.
First, we bound the maximal entry of $\mathcal{T}$.

\begin{lemma}[Bound of the max] \label{maxT}
Let $\mathcal{T} \in \mathcal{S}_p(N)$ satisfying $(ii)$ and $(iii)$. There exists a sequence of positive numbers $(\epsilon_N)_{N\geq 1}$ such that $\epsilon_N \rightarrow 0$ and
     \[ \mathbb{P}\left( \mathrm{max}_{i_1,\ldots,i_p} |\mathcal{T}_{i_1,\ldots,i_p}| \leq N \epsilon_N \right) = 1 - \mathcal{O} \left( \epsilon_N \right) \]
     In other words, $\mathrm{max}_{i_1,\ldots,i_p} |\mathcal{T}_{i_1,\ldots,i_p}| = o \left( N \right)$ with high probability.
\end{lemma}

\begin{proof}
As $X^p \in \mathcal{L}^1$, the family $\left(\mathcal{T}_{i_1,\ldots,i_p}^p\right)_{i_1,\ldots,i_p}$ is uniformly integrable. 
Hence we know thanks to the de la Vallée Poussin criterion \cite{chandra} that there exists $\phi: \R_+\rightarrow \R_+$ non-decreasing such that $ L(x):=\frac{\phi(x)}{x} { \underset{x \rightarrow +\infty}{\rightarrow}} +\infty$ and
$$ \mathrm{max}_{i_1,\ldots,i_p} \mathbb{E} \left[ \phi\left(\vert \mathcal{T}_{i_1,\ldots,i_p} \vert^p \right)\right] = \mathbb{E} \left[ \phi\left(\vert X \vert^p \right)\right] =: c < \infty. $$
For $\epsilon > 0$, we then have
\begin{align*}
    \mathbb{P}\left( \mathrm{max}_{i_1,\ldots,i_p} |\mathcal{T}_{i_1,\ldots,i_p}| > N \epsilon \right) 
    &\leq N^p \mathbb{P}\left( |X| > N \epsilon \right) \tag*{(union bound)} \\
    &= N^p \mathbb{P}\left( \phi \left(|X|^p \right) > \phi \left( N^p \epsilon^p \right) \right) \tag*{($\phi$ is non-decreasing)}  \\
    &\leq N^p \frac{\mathbb{E} \left[ \phi\left(\vert X \vert^p \right)\right]}{\phi \left( N^p \epsilon^p \right)} \tag*{(Markov's inequality)} \\
    &= \frac{c}{\epsilon^p L(N^p \epsilon^p)}
\end{align*}
Now for $k\in \N^*$ fixed, there exists $N_k$ such that $L(N_k^p / k^p) > k^{p+1}$. 
Take 
$$\epsilon_{N_k}=\ldots=\epsilon_{N_{k+1}-1}=\frac{1}{k}.$$
Then $\epsilon_N \rightarrow 0$ and for all $N$, $\epsilon_N^p L(N^p \epsilon_N^p) > (\epsilon_N)^{-1}$. 
Hence,
$$\mathbb{P}\left( \mathrm{max}_{i_1,\ldots,i_p} |\mathcal{T}_{i_1,\ldots,i_p}| > N \epsilon_N \right) \leq c \times \epsilon_N , $$
so $\mathbb{P}\left( \mathrm{max}_{i_1,\ldots,i_p} |\mathcal{T}_{i_1,\ldots,i_p}| \leq N \epsilon_N \right) = 1 -  \mathcal{O} \left( \epsilon_N \right)$ and then $\mathrm{max}_{i_1,\ldots,i_p} |\mathcal{T}_{i_1,\ldots,i_p}| = o \left( N \right)$ with high probability. This concludes the proof of Lemma \ref{maxT}.
\end{proof}

We come back to the initial proof. 
Let $\mathcal{T} \in \mathcal{S}_p(N)$ satisfying $(i)-(iv)$. 
We define 
\[ \widehat{\mathcal{T}}_{i_1,\ldots,i_p} := \mathcal{T}_{i_1,\ldots,i_p} \mathbb{1}_{| \mathcal{T}_{i_1,\ldots,i_p} | \leq N \epsilon_N}, \]
and we have by $(iv)$, $(i)$ and Lemma \ref{maxT}
$$\left\{
    \begin{array}{ll}
        \mathbb{E} \left[ \widehat{\mathcal{T}}_{i_1,\ldots,i_p} \right] = 0 \\
        \mathrm{Var}\left[ \widehat{\mathcal{T}}_{i_1,\ldots,i_p} \right] = \frac{1}{(p-1)!} \mathbb{P}\left(|\mathcal{T}_{i_1,\ldots,i_p}| \leq N \epsilon_N \right) \rightarrow \frac{1}{(p-1)!} \\
        \mathbb{P}\left(\exists i_1,\ldots,i_p : \widehat{\mathcal{T}}_{i_1,\ldots,i_p} \neq \mathcal{T}_{i_1,\ldots,i_p} \right) \leq \epsilon_N \quad  \quad (\star)
    \end{array}
\right.
$$
Let $n$ be a fixed integer. 
In the following we prove that 
$$\mathbb{E} \left[\widehat{a}_N \right] := \mathbb{E} \left[ \frac{1}{N} I_{2n} \left(\frac{\widehat{\mathcal{T}}}{N^{\frac{p-1}{2}}}\right)\right] = F_p(n) + \mathcal{O} \left( \frac{1}{N} \right).$$ 
Indeed, it is sufficient to obtain the convergence in probability of $a_N:=\frac{1}{N} I_{2n} \left(\frac{\mathcal{T}}{N^{\frac{p-1}{2}}}\right)$ as
$$ a_N - F_p(n) = \left(a_N -\widehat{a}_N \right) + \left( \widehat{a}_N - F_p(n) \right), $$
with $\mathbb{P} \left( \vert a_N -\widehat{a}_N \vert > 0 \right) \leq \epsilon_N $ by $(\star)$ and $\mathbb{P} \left( \vert \widehat{a}_N -F_p(n) \vert > \delta \right) = \mathcal{O} \left(\frac{1}{N \delta} \right) $ by Markov's inequality. 
Now recall that 
\[ \mathbb{E} \left[ \frac{1}{N} I_n \left(\frac{\widehat{\mathcal{T}}}{N^{\frac{p-1}{2}}}\right)\right] = N^{-(1 + \gamma n )} \sum_{\pi \in \mathcal{P}(np/2)} \sum_{ \genfrac{}{}{0pt}{}{1\leq a_1,\ldots,a_{\vert \pi \vert} \leq N}{\text{distinct}} } \prod_{e=(v_1,\ldots,v_p) \in E_\pi^*} \mathbb{E}\left[ \widehat{\mathcal{T}}_{a_{v_1} \ldots a_{v_p}}^{m(e)} \right]  \]
Since $\mathbb{E} \left[ \widehat{\mathcal{T}}_{i_1,\ldots,i_p} \right] = 0$, we still consider only the graphs with 
$$ \forall e \in E(b_\pi^\dagger), m(e) \geq 2, $$
and then
\begin{equation} 
    \vert E^*_\pi \vert \leq \frac{n}{2}
\end{equation} 
Now we write:
\[ m_{2-p} := | \{e \in E(b_\pi^\dagger) : 2 \leq m(e) \leq p \} | , \]
\[ m_{>p} := | \{e \in E(b_\pi^\dagger) : m(e) > p \} | . \]
We have immediately  
\[ m_{2-p} + m_{>p} = \vert E^*_\pi \vert \]
and by definition of $m_{2-p}$, $m_{>p}$, we know that 
\[ n = \sum_{e \in E(b_\pi^\dagger)} m(e) \geq 2 m_{2-p} + \sum_{e: m(e)>p} m(e) . \]
Then, these two relations give:
\begin{equation} \label{m>p}
    \sum_{e: m(e)>p} (m(e)-2) \leq n-2 \vert E^*_\pi \vert .
\end{equation}
Moreover, by $(iii)$, there exists $C>0$ such that if $m(e) \geq p$:
\[\mathbb{E} \left[ \widehat{\mathcal{T}}^{m(e)}_{i_1,\ldots,i_p}\right] = \mathbb{E} \left[\widehat{\mathcal{T}}^{p}_{i_1,\ldots,i_p} \widehat{\mathcal{T}}^{m(e)-p}_{i_1,\ldots,i_p} \right] \leq C \times (N \epsilon_N)^{m(e)-p} \]
Hence we have by Equation \eqref{m>p}
\begin{align*}
    \prod_{e : m(e)>p} \mathbb{E} \left[ \widehat{\mathcal{T}}^{m(e)}_{v_1 \ldots v_p}\right] &\leq C^{m_{>p}} (N\epsilon_N)^{\sum_{e : m(e)>p} (m(e) - p)} \\
    &\leq C' (\epsilon_N)^{m_{>p}} N^{n - 2 \vert E^*_\pi \vert - (p-2) m_{>p}} .
\end{align*}
Let $b$ be a combinatorial hypermap such that $m_{>p} \neq 0$. As $(\epsilon_N)^{m_{>p}} = o (1)$ and $\vert \pi \vert \leq 1 + (p-1) \vert E^*_\pi \vert$ (Equation \eqref{eq:2}), then its contribution is bounded by
\[ o \left( N^{-1 -\gamma n } N^{1 + (p-1) \vert E^*_\pi \vert} N^{n - 2 \vert E^*_\pi \vert - (p-2) m_{>p}}\right) .\]
We can bound the exponent by
\begin{align*}
    - \frac{p-3}{2} n + (p-3) \vert E^*_\pi \vert) - (p-2) m_{>p} = - \left(\underbrace{\frac{p-3}{2}\left(\frac{n}{2}- \vert E^*_\pi \vert\right)}_{\geq 0} + \underbrace{(p-2) m_{>p}}_{\geq 1 \text{ if $m_{>p} \neq 0$}} \right)  \leq -1
\end{align*}
So the asymptotic contribution of a combinatorial map with $m_{>p} \neq 0$ will still be $\mathcal{O}\left(\frac{1}{N}\right)$. 
\newline If $m_{>p} = 0$, then we are back in the case where the moments are bounded and the contribution is asymptotically given by the double hypertrees, the other ones having a contribution in $\mathcal{O}\left(\frac{1}{N}\right)$. We get the result.
\begin{remark}\label{rem:nec}
Some remarks about the necessity of conditions $(i)-(iv)$:
\begin{itemize}
    \item[$(i)$] Again only the variance of the off-diagonal entries need to be $\frac{1}{(p-1)!}$
    \item[$(ii)$] We may only require that the entries $\mathcal{T}_{i_1,\ldots,i_p}$, such that the tuples $(i_1,\ldots,i_p)$ have the same type, are identically distributed and all goes the same. The type is the number of blocks of each size in the partition into equal elements. Two tuples have the same type if they are equal after reordering (let say in increasing order) and bijective relabelling. 
    \item[$(iii)$] In the non-$i.i.d$ case, we need $p+1$ finite moments to ensure that the family $\left(\mathcal{T}_{i_1,\ldots,i_p}^p\right)_{i_1,\ldots,i_p}$ is uniformly integrable and then the same proof applies. 
    \item[$(iv)$] If the law of $\mathcal{T}_{i_1,\ldots,i_p}$ is no longer symmetric, we would have to define $$\widehat{\mathcal{T}}_{i_1,\ldots,i_p} := \frac{\mathcal{T}_{i_1,\ldots,i_p} \mathbb{1}_{| \mathcal{T}_{i_1,\ldots,i_p} | \leq N \epsilon_N} - \mathbb{E}\left[ \mathcal{T}_{i_1,\ldots,i_p} \mathbb{1}_{| \mathcal{T}_{i_1,\ldots,i_p} | \leq N \epsilon_N}\right]}{\kappa_{i_1,\ldots,i_p}}$$ to have a centered variable with right variance, then control $\eta := \mathbb{E}\left[ \mathcal{T}_{i_1,\ldots,i_p} \mathbb{1}_{| \mathcal{T}_{i_1,\ldots,i_p} | \leq N \epsilon_N}\right]$ and bound $\vert I_n(\widehat{\mathcal{T}}+\eta 1^{\otimes p}) - I_n(\widehat{\mathcal{T}}) \vert $. We will not try to give more details here.
\end{itemize}
\end{remark}

\subsection{Proof of Theorem \ref{thm:2}}

Let $p\geq3$. To prove Theorem \ref{thm:2}, it is sufficient to prove Lemma \ref{lem:var}. Indeed, if for all $b,b'$,
$$\mathbb{E} \left[ \mathrm{Tr}_{b}(\mathcal{W}_N)\mathrm{Tr}_{b'}(\mathcal{W}_N)\right] - \mathbb{E} \left[ \mathrm{Tr}_{b}(\mathcal{W}_N)\right] \mathbb{E}\left[\mathrm{Tr}_{b'}(\mathcal{W}_N)\right] = \mathcal{O}(1), $$
then
\begin{align*}
    \mathrm{Var} \left[ \frac{1}{N} I_n \left(\mathcal{W}_N\right)\right]
    &= \frac{1}{N^2} \mathbb{E} \left[ \left( \sum_b \mathrm{Tr}_{b}(\mathcal{W}_N) \right)^2\right] - \frac{1}{N^2} \mathbb{E} \left[ \sum_b \mathrm{Tr}_{b}(\mathcal{W}_N) \right]^2 \\
    &= \frac{1}{N^2}\sum_{b,b'} \left( \mathbb{E} \left[ \mathrm{Tr}_{b}(\mathcal{W}_N)\mathrm{Tr}_{b'}(\mathcal{W}_N)\right] - \mathbb{E} \left[ \mathrm{Tr}_{b}(\mathcal{W}_N)\right] \mathbb{E}\left[\mathrm{Tr}_{b'}(\mathcal{W}_N)\right] \right) \\
    &= \mathcal{O}\left(\frac{1}{N^2}\right).
\end{align*}

\begin{proof}[Proof of Lemma \ref{lem:var}]
We will firstly get a bound in $\mathcal{O}\left(N\right)$ which we will improve in a second time. 
Let $\mathcal{W}= \T / N^{\frac{p-1}{2}}$ be a Wigner tensor. 
Let $b$ and $d$ be two $p$-regular combinatorial maps with the same number of vertices and denote $n=\vert V(b)\vert = \vert E(b^\dagger) \vert=\vert V(d)\vert = \vert E(d^\dagger) \vert$. 
Denote also 
$$\kappa_{b,d} := \mathbb{E} \left[ \mathrm{Tr}_{b}(\mathcal{W}_N)\mathrm{Tr}_{d}(\mathcal{W}_N)\right] - \mathbb{E} \left[ \mathrm{Tr}_{b}(\mathcal{W}_N)\right] \mathbb{E}\left[\mathrm{Tr}_{d}(\mathcal{W}_N)\right].$$
Denote finally $\mathcal{M}_{2n}$ the set of melonic graphs with $2n$ vertices.
\begin{align*}
    \kappa_{b,d} &= N^{-2\gamma n} \sum_{\pi, \eta \in \mathcal{P}(np/2)} \mathbb{E} \left[ \mathrm{Tr}^0_{b_{\pi}}(\T) \mathrm{Tr}^0_{d_{\eta}}(\T) \right] - N^2 \alpha^2 \mathbb{1}_{G(b) \in \mathcal{M}_{2n}} \mathbb{1}_{G(d) \in \mathcal{M}_{2n}} + \mathcal{O}\left( N\right) \\
    &= N^{-2\gamma n} \sum_{\pi, \eta \in \mathcal{P}(np/2)} \sum_{ \genfrac{}{}{0pt}{}{ a_1,\ldots,a_{\vert \pi \vert} }{\text{distinct}} } \sum_{ \genfrac{}{}{0pt}{}{b_1,\ldots,b_{\vert \eta \vert} }{\text{distinct}} } \prod_{e=(v_1,\ldots,v_p) \in E_\pi^*} \prod_{e'=(v'_1,\ldots,v'_p) \in E_\eta^*} \mathbb{E}\left[ \mathcal{T}_{a_{v_1} \ldots a_{v_p}}^{m(e)} \mathcal{T}^{m(e')}_{b_{v'_1} \ldots b_{v'_p}} \right] \\
    & \quad - N^2 \alpha^2 \mathbb{1}_{G(b) \in \mathcal{M}_{2n}} \mathbb{1}_{G(d) \in \mathcal{M}_{2n}} + \mathcal{O}\left( N\right) .
\end{align*}
For $\pi, \eta \in \mathcal{P}(np/2)$, the hypergraph $H(b_{\pi}) \cup H(d_{\eta})$ (resp. $H(b_{\pi}^\dagger) \cup H(d_{\eta}^\dagger)$) has $2n$ vertices (resp. hyperedges) and at most two connected components. 
The crucial point is that the term $\mathbb{E} \left[ \mathcal{T}^{m(e)}_{v_1 \ldots v_p}  \mathcal{T}^{m(e')}_{b_{v'_1} \ldots b_{v'_p}} \right]$ is not null if and only if $m(e) \geq 2$ and $m(e') \geq 2$, or $e=e'$.

Again we denote $H^*$ the simple reduced hypergraph of $H=H(b_{\pi}^\dagger) \cup H(d_{\eta}^\dagger)$. Its number of vertices is denoted $M$, we have 
\begin{itemize}
    \item[$\bullet$] if $H^*$ has one connected component, $\vert E^* \vert \leq \frac{2n}{2} =n$ and then Lemma \ref{lem:euler} gives
    \[ M = |V^*| \leq 1 + (p-1) n = 1 + 2 \gamma n \]
    \item[$\bullet$] if $H^*$ has two connected component $H_1=H(b_{\pi}^\dagger)$ and $H_2=H(d_{\eta}^\dagger)$, 
    \[M = |V^*_1| + |V^*_2| \leq 2 + (p-1) (|E^*_1| + \vert E^*_2 \vert ) \] with $|E^*_1|, |E^*_2| \leq \frac{n}{2}$. Then, we have
    \[ M \leq 2 + 2 \gamma n \]
    with equality if and only if $H_1$ and $H_2$ are double hypertrees. 
\end{itemize}
Hence if $H$ is the disjoint union of two hypertrees, 
\[N^M = N^{2 + 2 \gamma n} ,\]
and in all the other cases,
\[N^M =  \mathcal{O}\left( N^{1 + 2\gamma n }\right) .\]
Hence we get 
\begin{align*}
    \kappa_{b,d} &= N^{-2 \gamma n} N^{2+2 \gamma n} \alpha^2 \mathbb{1}_{G(b) \in \mathcal{M}_{2n}} \mathbb{1}_{G(d) \in \mathcal{M}_{2n}} + \mathcal{O}\left( N^{-2 \gamma n} N^{1+2 \gamma n} \right) \\
    & \quad - N^2 \alpha^2 \mathbb{1}_{G(b) \in \mathcal{M}_{2n}} \mathbb{1}_{G(d) \in \mathcal{M}_{2n}} + \mathcal{O}\left( N \right) \\
    &= \mathcal{O}\left(N\right)
\end{align*}

We will now improve this bound. The first question is which $b_{\pi}$ are such that $\mathrm{Tr}^0_{b_\pi}$ is of order $1$. Remark that 
\begin{itemize}
    \item[-] if $n$ is odd, $\vert E^* \vert \leq \frac{n-1}{2}$ and then
    \[ \vert V^* \vert \leq 1 + (p-1) \frac{n-1}{2} = 1 + \gamma n - \frac{p-1}{2}.\]
     For odd $p$, $\mathcal{B}_n^{(p)} = \emptyset$ and for even $p$, we have $p\geq 4$ so any hypermap will contribute as $\mathcal{O}\left( \frac{1}{N}\right)$ because
    \[ \vert V^* \vert \leq 1 + \gamma n - \frac{3}{2} < \gamma n\]
    \item[-] if $n=2m$ is even, the hypermaps contributing at order $1$ are exactly the ones with one cycle in their reduced hypergraph. Denote $A_1(2m)$ the set of such hypermaps.
\end{itemize}
Hence, we have
\begin{align*}
    \mathbb{E} & \left[ \mathrm{Tr}_{b}(\mathcal{W}_N)\right]  \mathbb{E} \left[\mathrm{Tr}_{d}(\mathcal{W}_N)\right] = \\
    & N^2 \alpha^2 \mathbb{1}_{G(b),G(d) \in \mathcal{M}_{2n}} + N \alpha (\sum_{\pi \in \mathcal{P}(np/2)} \beta_{b_\pi} \mathbb{1}_{ \genfrac{}{}{0pt}{}{G(d) \in \mathcal{M}_{2n}}{b_\pi  \in A_1(2n)} } + \sum_{\eta \in \mathcal{P}(np/2)} \beta_{d_\eta} \mathbb{1}_{\genfrac{}{}{0pt}{}{G(b) \in \mathcal{M}_{2n}}{d_\pi  \in A_1(2n)} }) + \mathcal{O}\left( 1\right)
\end{align*}
But recall that
\begin{align*}
    \mathbb{E} \left[ \mathrm{Tr}_{b}(\mathcal{W}_N) \mathrm{Tr}_{d}(\mathcal{W}_N)\right] 
    = N^{-2\gamma n} \sum_{\pi, \eta } \sum_{ \genfrac{}{}{0pt}{}{ a_1,\ldots,a_{\vert \pi \vert} }{\text{distinct}} } \sum_{ \genfrac{}{}{0pt}{}{b_1,\ldots,b_{\vert \eta \vert} \leq N}{\text{distinct}} } \prod_{e\in E_\pi^*} \prod_{e' \in E_\eta^*} \mathbb{E}\left[ \mathcal{T}_{a_{v_1} \ldots a_{v_p}}^{m(e)} \mathcal{T}^{m(e')}_{b_{v'_1} \ldots b_{v'_p}} \right] .
\end{align*} 
 Firstly, if $H_1$ and $H_2$ are disconnected, we can have:
\begin{itemize}
    \item[$\bullet$] $H_1, H_2 \in \mathcal{M}_{2n}$ which gives the term $N^2 \alpha^2 \mathbb{1}_{G(b),G(d) \in \mathcal{M}_{2n}}$,
    \item[$\bullet$] $H_1\in \mathcal{M}_{2n}$ and $H_2 \in A_1(2n)$ or $H_2\in \mathcal{M}_{2n}$ and $H_1 \in A_1(2n)$ which gives the term $N \alpha (\sum_{\pi \in \mathcal{P}(np/2)} \beta_{b_\pi} \mathbb{1}_{\genfrac{}{}{0pt}{}{G(d) \in \mathcal{M}_{2n}}{b_\pi  \in A_1(2n)} } + \sum_{\eta \in \mathcal{P}(np/2)} \beta_{d_\eta} \mathbb{1}_{\genfrac{}{}{0pt}{}{G(b) \in \mathcal{M}_{2n}}{d_\pi  \in A_1(2n)} })$,
    \item[$\bullet$] the other terms give immediately a contribution in $\mathcal{O}\left( 1\right)$.
\end{itemize}
Secondly, if $H_1$ and $H_2$ are connected $i.e.$ there exists 
$$ e=e'= \{v_1,\ldots,v_p \} \in E_1 \cap E_2. $$ 
This hypergraph will contribute at order $N$ if and only if $M$ reach the maximal bound $1+ 2\gamma n $, $i.e.$ 
\[H = H_1 \cup H_2 \text{ is a double hypertree.} \] 
This will never be possible. 
Indeed, we must have $m_H(e)=2$, so $m_{H_1}(e) = m_{H_2}(e) =1$ but since $b$ is $2$-regular $v_1$ must belong to an even number of hyperedges in $H_1=H(b_\pi^\dagger)$ and hence all the other vertices cannot have multiplicity $2$ in $H_1$. 
So we must have another hyperedge $f$ distinct from $e$ satisfying $v_1 \in f$ and $m_{H_1}(f) = m_{H_2}(f) =1$. Then, as $H^*$ is a hypertree, one can pick another vertex distinct from $v_1$ in $f$ and repeat independently the previous argument. Since the hypergraph is finite, that is a contradiction. 

This concludes the proof.
\end{proof}

\subsection{The universal law}\label{sec:mes}

The Fuss-Catalan numbers are the moments of a free Bessel law. 
We derive in this section some results about the limiting measure $\mu_{\infty}$. 
We consider the generating series $T_p$ satisfying
\[ T_p(z) = 1 + z T_p(z)^p . \]
The relation gives $ z(T_p) = T_p^{1-p} - T_p^{-p} $ and hence
\[ \frac{\partial z}{\partial T_p}=0 \Leftrightarrow 0=(1-p)T_p + p \Leftrightarrow T_p = \frac{p}{p-1} .\]
By the implicit function theorem the equation can be solved up to the critical point  
\[ z_c = \left(\frac{p}{p-1}\right)^{1-p} - \left(\frac{p}{p-1}\right)^{-p} = \frac{(p-1)^{p-1}}{p^p} . \]
For $|z|<z_c$, it admits the absolute convergent series representation 
\[ T_p(z) = \sum_{k\geq 0} F_p(k) z^k ,\]
where $ F_p(k) = \frac{1}{pk+1}\binom{pk+1}{k}$. 
Now we define 
\[ \mathcal{R}_{\infty}(z) := \sum_{n \geq 0} \frac{1}{z^{n+1}} \times  \mathbb{1}_{\text{n even }} F_p\left( \frac{n}{2} \right), \]
that is simply 
\[ \mathcal{R}_{\infty}(z) = \frac{1}{z} \sum_{n \geq 0} \frac{1}{z^{2n}}  F_p\left( n \right) = \frac{1}{z} T_p \left(\frac{1}{z^2} \right) . \]
It is the Cauchy-Stieltjes transform of a measure $\mu_{\infty}^{(p)}$ \cite{collins}.

\begin{remark}
We have the identity 
    \[ z\mathcal{R}_{\infty}(z) = T_p \left(\frac{1}{z^2}  \right) = 1 + \frac{1}{z^2} T_p \left(\frac{1}{z^2}  \right)^p = 1 + z^{p-2} \mathcal{R}_{\infty}(z)^p\]
    so we find an equation which generalizes the one known for the matrices,
    \[ z^{p-2} \mathcal{R}_{\infty}(z)^p - z\mathcal{R}_{\infty}(z) +1 = 0 \]
\end{remark}

Denoting $z_c=\frac{(p-1)^{p-1}}{p^p}$ as previously, the Fuss–Catalan numbers admit the integral representation
\[ F_p(n) = \int_0^{1/z_c} x^n P_p(x) dx , \]
with $P_p$ a real positive function. 
Indeed, as proved in \cite{pens} it can be written in terms of hypergeometric functions (or with Meijer G-function),
\[ P_p(x) = \sum_{k=1}^{p-1} \Lambda_{k,p} \ \ {}_{p-1}F_{p-2} \left( \left\{ 1 - \frac{1+j}{p-1} + \frac{k}{p} \right\}_{j=1}^{p-1}, \left\{1 + \frac{k-j}{p} \right\}_{\genfrac{}{}{0pt}{}{j=1}{j \neq k} }^{p-1} ; z_c x \right) , \]
where 
\[ \Lambda_{k,p} = \frac{1}{(p-1)^{3/2}} \sqrt{\frac{p}{2 \pi}} z_c^{\frac{k}{p}} \frac{\prod_{1\leq j \leq p-1}^{j\neq k} \Gamma\left(\frac{j-k}{p}\right)}{\prod_{1 \leq j \leq p-1} \Gamma\left(\frac{j+1}{p-1}-\frac{k}{p} \right)} .\]
Then $ T_p(z) = \sum_{k\geq 0} F_p(k) z^k$ (absolutely convergent for $|z|<z_c$) can be analytically continued on $\mathbb{C} \setminus [ z_c, \infty )$ and admits the convergent integral representation
\begin{align*}
    T_p(z) &= \sum_{k\geq 0} z^k \int_0^{1/z_c} x^k P_p(x) dx \\
    &= \int_0^{1/z_c} \frac{1}{1-zx} P_p(x) dx .
\end{align*} 
Hence,
\begin{align*}
    \mathcal{R}_{\infty}(z) &= \sum_{n\geq 0} \frac{1}{z^{n+1}} \mathbb{1}_{\text{n even }} \int_0^{1/z_c} x^{n/2} P_p(x) dx \\
    &= \sum_{n\geq 0} \frac{1}{z^{n+1}} \mathbb{1}_{\text{n even }} 2 \int_0^{\sqrt{1/z_c}} y^n P_p(y^2) y dy \\
    &= \sum_{n\geq 0} \frac{1}{z^{n+1}} \mathbb{1}_{\text{n even }} \int_{-\sqrt{1/z_c}}^{\sqrt{1/z_c}} y^n P_p(y^2) |y| dy
\end{align*}
We denote 
\[ \omega_c := \sqrt{1/z_c} .\]
As $y \mapsto  y^n P_p(y^2) |y|$ is an odd function for $n$ odd, we get:
\begin{align*}
    \mathcal{R}_{\infty}(z) &= \sum_{n\geq 0} \frac{1}{z^{n+1}} \int_{-\omega_c}^{\omega_c} y^n |y| P_p(y^2) dy \\
    &= \frac{1}{z} \int_{-\omega_c}^{\omega_c} \frac{1}{1-\frac{y}{z}} |y| P_p(y^2) dy \\
    &= \int_{-\omega_c}^{\omega_c} \frac{1}{z-y} |y| P_p(y^2) dy .
\end{align*}
Finally, we get 
\begin{equation} \notag
    \mu_{\infty}(y) = |y| P_p(y^2)
\end{equation}
with a support on $[-\omega_c,\omega_c]$.

\begin{remark}
    For $p=2$, we have $z_c = \frac{1}{4}$ so $\omega_c=2$,
    \[ P_2(x) = \frac{\sqrt{1-\frac{x}{4}}}{\pi \sqrt{x}}\]
    and then we find the Wigner semi-circle law, for $y \in [-2,2]$,
    \[ \mu^{(2)}_{\infty}(y) = \frac{1}{2\pi}\sqrt{4-y^2} \]
\end{remark}

\begin{remark}
    For p=3, we can compute 
\[ \omega_c = \sqrt{\frac{3^3}{2^2}} \simeq 2.598 \]
\[ P_3(x) = \frac{1}{2 \pi x^{2/3}} \frac{3^{1/2}}{2^{1/3}} \frac{\left(1+\sqrt{1-\frac{4x}{27}} \right)^{2/3} - \left(\frac{4x}{27}\right)^{1/3}}{\left(1+\sqrt{1-\frac{4x}{27}} \right)^{1/3}}\]
Then for $y \in [-\omega_c,\omega_c]$,
\begin{align*}
    \mu^{(3)}_{\infty}(y) &= |y|  \frac{1}{2 \pi y^{4/3}} \frac{3^{1/2}}{2^{1/3}} \frac{\left(1+\sqrt{1-\frac{4y^2}{27}} \right)^{2/3} - \left(\frac{4y^2}{27}\right)^{1/3}}{\left(1+\sqrt{1-\frac{4y^2}{27}} \right)^{1/3}} \\
    &= \frac{\sqrt{3}}{2^{4/3} \pi |y|^{1/3}} \left( \left(1+\sqrt{1-\frac{4y^2}{27}} \right)^{1/3} -  \left(\frac{\left(\frac{4y^2}{27}\right)}{1+\sqrt{1-\frac{4y^2}{27}}} \right)^{1/3} \right) \\
    &= \frac{\sqrt{3}}{2^{4/3} \pi |y|^{1/3}} \left( \left(1+\sqrt{1-\frac{4y^2}{27}} \right)^{1/3} -  \left(1-\sqrt{1-\frac{4y^2}{27}} \right)^{1/3} \right) 
\end{align*} 
This measure has the profile given in Figure \ref{fig_Wig}.
\end{remark}

\subsection{Stability of the limit law under contraction}

Let us first recall the result that we are going to prove.
Let $p\geq 3$, $\mathcal{T} \in \mathcal{S}_p(N)$ be a symmetric tensor with independent Gaussian entries $\T_{i_1,\ldots,i_p} \sim \mathcal{N}(0, \sigma_{i_1,\ldots,i_p}^2)$ ($i.e.$ $\T / N^{\frac{p-1}{2}}$ belongs to the \textit{GOTE}) and let $u \in \mathbb{S}^{N-1}$ be a sequence of deterministic unit vectors. 
For $k\leq p-2$, we define 
$$\widetilde{T} := \frac{1}{N^{\frac{p-k-1}{2}}} \mathcal{T} \cdot u^k \in \mathcal{S}_{p-k}(N) $$ 
the normalized contraction of $\mathcal{T}$ by $u^{\otimes k}$. 
Then, for all $n \geq 0$,
\[ \frac{1}{N} I_n(\widetilde{T}) \underset{N \longrightarrow \infty}{\rightarrow} \int \lambda^n \widetilde{\mu}_{\infty}(d \lambda) , \]
where \[ \widetilde{\mu}_{\infty}(y) = \sqrt{\binom{p-1}{k}} \mu^{(p-k)}_{\infty}\left(y\sqrt{\binom{p-1}{k}}\right) , \]
is the dilated Wigner-Gurau law of order $p-k$ supported on 
\[ \left[-\sqrt{\frac{(p-k)^{p-k}}{\binom{p-1}{k}(p-k-1)^{p-k-1}}}, \sqrt{\frac{(p-k)^{p-k}}{\binom{p-1}{k}(p-k-1)^{p-k-1}}} \right] . \]

This result makes the link with the approach by contractions of the tensors proposed for instance by Couillet, Comon, Goulart in \cite{couillet}. 
Indeed we prove here that even if the tensor contracted does not have independent entries anymore, we have convergence to the moments of the same law at order $p-k$ with a particular scaling. 
Before proving the Theorem, we make more precise the link with \cite{couillet} in the following remark.

\begin{remark}[Case $k=p-2$] \label{rem-cou}
    In the particular case $k=p-2$ where the contraction is a matrix, we find back Theorem 2 of \cite{couillet}. 
    \newline Indeed, their normalization for the \textit{GOTE} differs from our one by a factor $p$ with $\mathbb{E} \left[ \mathcal{T}^2_{1\ldots p}\right] = \frac{1}{p!}$ and not $\frac{1}{(p-1)!}$ as for us. 
    We chose our normalization because it gives the classical matrix \textit{GOE} in the case $p=2$. 
    Hence to find their result from ours, the limit law has to be dilated by a factor $\sqrt{p}$. So, for $\mathcal{W}$ in their \textit{GOTE}, $\frac{1}{\sqrt{N}}\mathcal{W} \cdot u^{p-2}$ is a matrix with spectral measure $\mu_{\frac{1}{\sqrt{N}}\mathcal{W} \cdot u^{p-2}}$ and by Theorem \ref{thm:contr} and Theorem \ref{thm:2}, we have
    \[ \mu_{\frac{1}{\sqrt{N}}\mathcal{W} \cdot u^{p-2}} \rightarrow \rho \]
    where $\rho$ is a semi-circular law with density:
    \begin{align*}
        \rho(dx) &= \sqrt{p} \sqrt{\binom{p-1}{p-2}}\mu^{(2)}_{\infty}(dx\sqrt{p}\sqrt{p-1}) \\
        &= \frac{p(p-1)}{2 \pi} \sqrt{\left( \frac{4}{p(p-1)} -x^2 \right)^+} dx
    \end{align*}
    and a support on $\left[ -\frac{2}{\sqrt{p(p-1)}}, \frac{2}{\sqrt{p(p-1)}} \right]$. 
    It is exactly what they proved.
\end{remark}

\begin{proof}[Proof of Theorem \ref{thm:contr}]
We fix $k\leq p-2$. As we are only in the Gaussian case, by orthogonal invariance we can assume that $u = e^{(1)}$ is the first vector of the canonical basis of $\mathbb{R}^N$. Recall that
\begin{align*}
    \mathbb{E} &\left[ \frac{1}{N} I_n \left(\frac{\mathcal{T} \cdot \left( e^{(1)} \right)^k}{N^{\frac{p-k-1}{2}}}\right)\right] = N^{-(\frac{p-k-1}{2} n +1)} \times \\ 
    & \hspace{2cm} \sum_{b \in \mathcal{B}^{(p-k)}_{n}} \sum_{\pi \in \mathcal{P}(n(p-k)/2)} \sum_{ \genfrac{}{}{0pt}{}{ a_1,\ldots,a_{\vert \pi \vert}}{\text{distinct}} } \mathbb{E}\left[ \prod_{e=(v_1,\ldots,v_p) \in E(b_\pi^\dagger)} \mathcal{T} \cdot \left( e^{(1)} \right)^k_{a_{v_1} \ldots a_{v_p}} \right] .
\end{align*}  
We also have immediately that
\[ \left( \mathcal{T} \cdot \left( e^{(1)} \right)^k\right) _{a_{v_1} \ldots a_{v_{p-k}}} = \mathcal{T}_{\underbrace{1\ldots 1}_{k \text{ times}} a_{v_1} \ldots a_{v_{p-k}}} . \]
Then considering if one of the $a_j$ match with $1$ or not, we can write 
\begin{align*}
    \mathbb{E} &\left[ \frac{1}{N} I_n \left(\frac{\mathcal{T} \cdot \left( e^{(1)} \right)^k}{N^{\frac{p-k-1}{2}}}\right)\right] = N^{-(1 + \frac{p-k-1}{2} n )} \sum_{b } \sum_{\pi } \sum_{ \genfrac{}{}{0pt}{}{2\leq a_1,\ldots,a_{\vert \pi \vert} \leq N}{\text{distinct}} } \mathbb{E}\left[ \prod_{e \in E(b_\pi^\dagger)} \mathcal{T}_{1 \ldots 1 a_{v_1} \ldots a_{v_{p-k}}} \right] \\
    & \hspace{1cm} + N^{-(1 + \frac{p-k-1}{2} n )} \sum_{b } \sum_{\pi } \vert \pi \vert \sum_{ \genfrac{}{}{0pt}{}{1= a_1 < a_2,\ldots,a_{\vert \pi \vert} \leq N}{\text{distinct}} } \mathbb{E}\left[ \prod_{e \in E(b_\pi^\dagger)} \mathcal{T}_{1 \ldots 1 a_{v_1} \ldots a_{v_{p-k}}} \right] \\
    & \hspace{4cm} =: A + B .
\end{align*}
Now we may apply the same arguments of the proof of Theorem \ref{thm:1},
\begin{align*}
    A &= N^{-(1 + \frac{p-k-1}{2} n )} \sum_{b } \sum_{\pi } \sum_{ \genfrac{}{}{0pt}{}{2\leq a_1,\ldots,a_{\vert \pi \vert} \leq N}{\text{distinct}} } \mathbb{E}\left[ \prod_{e \in E(b_\pi^\dagger)} \mathcal{T}_{1 \ldots 1 a_{v_1} \ldots a_{v_{p-k}}} \right] \\
    &= N^{-(1+\frac{p-k-1}{2} n )} N^{1 + \frac{p-k-1}{2} n} \mathbb{1}_{\text{n even }} F_{p-k}\left( \frac{n}{2} \right)  \left[(p-k-1)!\right]^{\frac{n}{2}} \mathbb{E} \left[ \mathcal{T}^2_{1\ldots 1 2\ldots (p-k+1)}\right]^{\frac{n}{2}} + \mathcal{O}\left( \frac{1}{N}\right)\\
    &= \mathbb{1}_{\text{n even }} F_{p-k}\left( \frac{n}{2} \right) \left[\frac{(p-k-1)!p}{p(p-1)\ldots (k+1)}\right]^{\frac{n}{2}} + \mathcal{O}\left( \frac{1}{N}\right) \\
    &= \mathbb{1}_{\text{n even }} F_{p-k}\left( \frac{n}{2} \right) \binom{p-1}{k}^{-\frac{n}{2}} + \mathcal{O}\left( \frac{1}{N}\right) ,
\end{align*}
and
\begin{align*}
    B = N^{-(1+ \frac{p-k-1}{2} n )} \mathcal{O}\left(N^{1+ \frac{p-k-1}{2} n -1}\right) = \mathcal{O}\left( \frac{1}{N}\right) .
\end{align*}
Denoting $\mathcal{R}_{\infty}(z) := \sum_{n \geq 0} \frac{1}{z^{n+1}} \binom{p-1}{k}^{-\frac{n}{2}} \mathbb{1}_{\text{n even }} F_{p-k}\left( \frac{n}{2} \right)$, we get 
\begin{align*}
    \mathcal{R}_{\infty}(z) &= \frac{1}{z}\sum_{n\geq 0} \frac{1}{\left( z\sqrt{\binom{p-1}{k}}\right) ^{n}} \int_{-\omega_c}^{\omega_c} y^n |y| P_p(y^2) dy \\
    &= \int_{-\omega_c / \sqrt{\binom{p-1}{k}}}^{\omega_c /\sqrt{\binom{p-1}{k}}} \frac{1}{z-y} \binom{p-1}{k} |y| P_p\left(\binom{p-1}{k}y^2\right) dy 
\end{align*}
Hence we finally obtain in this case 
\[ \widetilde{\mu}_{\infty}(y) = \sqrt{\binom{p-1}{k}} \mu^{(p-k)}_{\infty}\left(y\sqrt{\binom{p-1}{k}}\right) , \]
this concludes the proof.
\end{proof}

\appendix

\section{Analogy in the matrix case}\label{appendixA}

Let $N \in \mathbb{N}$ and $\mathcal{M} \in \mathcal{S}_2(N)$ be a symmetric matrix. 

\begin{definition}
    For $ z \in \mathbb{C} \setminus \left(\mathrm{Sp}(\mathcal{M}) \cup \{0 \} \right) $, we define:
    \begin{equation} \notag
    \mathcal{R}(z) := \frac{z^{-1}}{\Xi(z)}\int \frac{\|\phi\|^2}{N} \exp \left(-\left( \frac{1}{2}\|\phi \|^2 - \frac{1}{z}\frac{\mathcal{M} \cdot \phi^2}{2}\right)\right) [d\phi ] ,
    \end{equation}
    where $\phi = (\phi_1,\ldots,\phi_N) \in \mathbb{R}^N$, $[d\phi ] := (2 \pi)^{-N/2} \prod_{i=1}^N d\phi_i$ and $\mathcal{M} \cdot \phi^2= \sum_{1\leq i,j\leq N} \mathcal{M}_{ij} \phi_i \phi_j$.
\end{definition}

We prove that it is the usual resolvent trace for matrices, or in other words,

\begin{proposition}\label{prop:mat}
for all $ z \notin \mathrm{Sp}(\mathcal{M}) \cup \{0 \}$,
$$ \mathcal{R}(z) = \frac{1}{N} \mathrm{Tr}\left( (z\mathcal{I}-\mathcal{M})^{-1}\right) . $$
\end{proposition}

\begin{proof}
First, we define for $ z \notin \mathrm{Sp}(\mathcal{M}) \cup \{0 \}$,
    \[ \tilde{\Xi}(z) := \int \exp \left(-\left( \frac{1}{2}\|\phi \|^2 - \frac{1}{z}\frac{\mathcal{M} \cdot \phi^2}{2}\right)\right) \prod_{i=1}^N d\phi_i . \]
We write this expression as
\begin{align*}
    \tilde{\Xi}(z) &= \int \exp \left(-\frac{1}{2} \langle \phi, (\mathcal{I} - \frac{1}{z} \mathcal{M}) \phi \rangle \right) \prod_{i=1}^N d\phi_i  \\
    &= \int \exp \left(-\frac{1}{2} \langle \phi, \Sigma^{-1} \phi \rangle \right) \prod_{i=1}^N d\phi_i 
\end{align*} 
where $\Sigma := (\mathcal{I} - \frac{1}{z} \mathcal{M})^{-1}$ is a normal matrix. By an orthogonal change of basis in the integral, we can assume for the following that $\mathcal{M} = \mathrm{diag}(\lambda_1,\ldots,\lambda_N)$.

\begin{remark}
    For $z \in \mathbb{R}$ sufficiently large, $\Sigma$ is positive-definite and we have 
    \[ \tilde{\Xi}(z) = (2 \pi)^{N/2} \sqrt{\mathrm{det} \Sigma} .\]
    Hence, we have in this case
    \begin{align*}
    \frac{d}{dz}\left( \mathrm{ln} \left((2 \pi)^{-N/2} \tilde{\Xi}(z)\right) \right)
    &= \frac{1}{2}\frac{d}{dz}\left( -\sum_{i=1}^N \mathrm{ln} \left(1-\frac{\lambda_i}{z}\right)\right) \\
    &= \frac{-1}{2 z^2}\sum_{i=1}^N \lambda_i \left(1-\frac{\lambda_i}{z}\right)^{-1} \\
    &= \frac{-1}{2 z^2} \mathrm{Tr}(\mathcal{M} \Sigma)
\end{align*}
\end{remark}

This result holds in more generality for all $z \notin \mathrm{Sp}(\mathcal{M}) \cup \{0 \}$, as we are going to prove in the following Lemma. 
Denote $[d\phi ] := (2 \pi)^{-N/2} \prod_{i=1}^N d\phi_i$ and 
\[ \Xi(z) := \int \exp \left(-\left( \frac{1}{2}\|\phi \|^2 - \frac{1}{z}\frac{\mathcal{M} \cdot \phi^2}{2}\right)\right) [d\phi ] = (2 \pi)^{-N/2} \tilde{\Xi}(z). \]

\begin{lemma}\label{lem:log}
For all $ z \notin \mathrm{Sp}(\mathcal{M}) \cup \{0 \}$,
    \[ \frac{d}{dz}\left( \mathrm{ln} \Xi(z) \right) = \frac{-1}{2 z^2} \mathrm{Tr}(\mathcal{M} \Sigma). \]
\end{lemma}
\begin{proof}[Proof of Lemma \ref{lem:log}]
Let us compute for $ z \notin \mathrm{Sp}(\mathcal{M}) \cup \{0 \}$,
    \begin{align*}
        \frac{d}{dz}\left( \mathrm{ln} \Xi(z) \right) 
        &= \frac{-1}{z^2} \frac{1}{\Xi(z)} \int \frac{\mathcal{M} \cdot \phi^2}{2} \exp \left(-\frac{1}{2} \langle \phi, \Sigma^{-1} \phi \rangle\right) [d\phi ] \\
        &= \frac{-1}{2 z^2} \frac{1}{\Xi(z)} \sum_{i=1}^N \lambda_i \int \phi_i^2 \exp \left(-\sum_{j=1}^N \frac{1}{2 \left(1-\frac{\lambda_j}{z}\right)^{-1}}  \phi_j^2 \right) [d\phi] \\
        &= \frac{-1}{2 z^2} \frac{1}{\Xi(z)} \sum_{i=1}^N \lambda_i \left(1-\frac{\lambda_i}{z}\right)^{-1} \Xi(z) \\
    \end{align*}
That gives the desired result.
\end{proof}

Moreover, integrating by parts $\Xi(z)$, we have for $1 \leq i \leq N$:
\[ \Xi(z) = - \int \phi_i \left( -\phi_i + \frac{1}{z} \sum_{j=1}^N \mathcal{M}_{ij}\phi_j \right) \exp \left(-\left( \frac{1}{2}\|\phi \|^2 - \frac{1}{z}\frac{\mathcal{M} \cdot \phi^2}{2}\right)\right) [d\phi ] . \]
Hence summing on $1 \leq i \leq N$,
\begin{equation} \label{eq:partition}
     N \Xi(z) = \int \left( \|\phi\|^2 - \frac{1}{z} \mathcal{M} \cdot \phi^2 \right) \exp \left(-\left( \frac{1}{2}\|\phi \|^2 - \frac{1}{z}\frac{\mathcal{M} \cdot \phi^2}{2}\right)\right) [d\phi ] .
\end{equation}
We are now ready to conclude the proof of Proposition \ref{prop:mat}
\begin{align*}
    \mathcal{R}(z) &= z^{-1} + \frac{1}{N}\frac{z^{-1}}{\Xi(z)} \int \frac{1}{z} \mathcal{M} \cdot \phi^2 \exp \left(-\left( \frac{1}{2}\|\phi \|^2 - \frac{1}{z}\frac{\mathcal{M} \cdot \phi^2}{2}\right)\right) [d\phi ] \tag*{(Equation \eqref{eq:partition})} \\
    &= z^{-1} - \frac{1}{N}\frac{2}{\Xi(z)} \frac{d}{dz} \left( \Xi(z) \right) \\
    &= z^{-1} + \frac{1}{N}\frac{1}{z^2} \mathrm{Tr}(\mathcal{M} \Sigma) \tag*{(Lemma \ref{lem:log})} \\
    &= \frac{1}{N} \mathrm{Tr}\left(\frac{1}{z} \mathcal{I} + \frac{1}{z}\mathcal{M}(z\mathcal{I}-\mathcal{M})^{-1}\right) \\
    &= \frac{1}{N} \mathrm{Tr}\left( (z\mathcal{I}-\mathcal{M})^{-1}\right) . \tag*{(resolvent identity)}
\end{align*}

\end{proof}

\printbibliography

@article {couillet,
    AUTHOR = {Goulart, Jos\'e{} Henrique de Morais and Couillet, Romain and
              Comon, Pierre},
     TITLE = {A random matrix perspective on random tensors},
   JOURNAL = {J. Mach. Learn. Res.},
  FJOURNAL = {Journal of Machine Learning Research (JMLR)},
    VOLUME = {23},
      YEAR = {2022},
     PAGES = {Paper No. [264], 36},
      ISSN = {1532-4435,1533-7928},
   MRCLASS = {65F55 (15A69 15B52 60B20)},
  MRNUMBER = {4577703},
}

@article {cart,
    AUTHOR = {Cartwright, Dustin and Sturmfels, Bernd},
     TITLE = {The number of eigenvalues of a tensor},
   JOURNAL = {Linear Algebra Appl.},
  FJOURNAL = {Linear Algebra and its Applications},
    VOLUME = {438},
      YEAR = {2013},
    NUMBER = {2},
     PAGES = {942--952},
      ISSN = {0024-3795,1873-1856},
   MRCLASS = {15A69},
  MRNUMBER = {2996375},
MRREVIEWER = {Tan\ Zhang},
       DOI = {10.1016/j.laa.2011.05.040},
       URL = {https://doi.org/10.1016/j.laa.2011.05.040},
}

@article {breid,
    AUTHOR = {Breiding, Paul},
     TITLE = {How many eigenvalues of a random symmetric tensor are real?},
   JOURNAL = {Trans. Amer. Math. Soc.},
  FJOURNAL = {Transactions of the American Mathematical Society},
    VOLUME = {372},
      YEAR = {2019},
    NUMBER = {11},
     PAGES = {7857--7887},
      ISSN = {0002-9947,1088-6850},
   MRCLASS = {15A18 (15A69 60B20 62H99)},
  MRNUMBER = {4029684},
MRREVIEWER = {Vladislav\ Kargin},
       DOI = {10.1090/tran/7910},
       URL = {https://doi.org/10.1090/tran/7910},
}

@article {jagannath,
    AUTHOR = {Jagannath, Aukosh and Lopatto, Patrick and Miolane, L\'eo},
     TITLE = {Statistical thresholds for tensor {PCA}},
   JOURNAL = {Ann. Appl. Probab.},
  FJOURNAL = {The Annals of Applied Probability},
    VOLUME = {30},
      YEAR = {2020},
    NUMBER = {4},
     PAGES = {1910--1933},
      ISSN = {1050-5164,2168-8737},
   MRCLASS = {62F05 (62F10 62H25 82B26 82D30)},
  MRNUMBER = {4132641},
       DOI = {10.1214/19-AAP1547},
       URL = {https://doi.org/10.1214/19-AAP1547},
}

@book {gur,
    AUTHOR = {Gurau, Razvan},
     TITLE = {Random tensors},
 PUBLISHER = {Oxford University Press, Oxford},
      YEAR = {2017},
     PAGES = {x+333},
      ISBN = {978-0-19-878793-8},
   MRCLASS = {15-02 (05C15 05C90 15A69 60B99 60D05)},
  MRNUMBER = {3616422},
}

@article {qi,
    AUTHOR = {Qi, Liqun},
     TITLE = {Eigenvalues of a real supersymmetric tensor},
   JOURNAL = {J. Symbolic Comput.},
  FJOURNAL = {Journal of Symbolic Computation},
    VOLUME = {40},
      YEAR = {2005},
    NUMBER = {6},
     PAGES = {1302--1324},
      ISSN = {0747-7171,1095-855X},
   MRCLASS = {15A18 (15A15 15A69)},
  MRNUMBER = {2178089},
MRREVIEWER = {John\ Chollet},
       DOI = {10.1016/j.jsc.2005.05.007},
       URL = {https://doi.org/10.1016/j.jsc.2005.05.007},
}

@incollection {anand,
    AUTHOR = {Anandkumar, Anima and Ge, Rong and Hsu, Daniel and Kakade,
              Sham M. and Telgarsky, Matus},
     TITLE = {Tensor decompositions for learning latent variable models (a
              survey for {ALT})},
 BOOKTITLE = {Algorithmic learning theory},
    SERIES = {Lecture Notes in Comput. Sci.},
    VOLUME = {9355},
     PAGES = {19--38},
 PUBLISHER = {Springer, Cham},
      YEAR = {2015},
%      ISBN = {978-3-319-24486-0; 978-3-319-24485-3},
   MRCLASS = {62H12 (15A69 62F12 62H25 68Q32)},
  MRNUMBER = {3480733},
       DOI = {10.1007/978-3-319-24486-0\_2},
       URL = {https://doi.org/10.1007/978-3-319-24486-0_2},
}

@article {land,
    AUTHOR = {Landsberg, Joseph M. and Qi, Yang and Ye, Ke},
     TITLE = {On the geometry of tensor network states},
   JOURNAL = {Quantum Inf. Comput.},
  FJOURNAL = {Quantum Information \& Computation},
    VOLUME = {12},
      YEAR = {2012},
    NUMBER = {3-4},
     PAGES = {346--354},
      ISSN = {1533-7146},
   MRCLASS = {81P16},
  MRNUMBER = {2933533},
}

@article {sidi,
    AUTHOR = {Sidiropoulos, Nicholas D. and De Lathauwer, Lieven and Fu,
              Xiao and Huang, Kejun and Papalexakis, Evangelos E. and
              Faloutsos, Christos},
     TITLE = {Tensor decomposition for signal processing and machine
              learning},
   JOURNAL = {IEEE Trans. Signal Process.},
  FJOURNAL = {IEEE Transactions on Signal Processing},
    VOLUME = {65},
      YEAR = {2017},
    NUMBER = {13},
     PAGES = {3551--3582},
      ISSN = {1053-587X,1941-0476},
   MRCLASS = {94A12},
  MRNUMBER = {3666587},
MRREVIEWER = {Yulin\ Zhang},
       DOI = {10.1109/TSP.2017.2690524},
       URL = {https://doi.org/10.1109/TSP.2017.2690524},
}

@misc{gur1,
      title={On the generalization of the {W}igner semicircle law to real symmetric tensors}, 
      author={Razvan Gurau},
      year={2020},
      eprint={2004.02660},
      archivePrefix={arXiv},
      primaryClass={math-ph},
      url={https://arxiv.org/abs/2004.02660}, 
}

@article{gur2,
   title={Quenched equals annealed at leading order in the colored {SYK} model},
   volume={119},
   ISSN={1286-4854},
   url={http://dx.doi.org/10.1209/0295-5075/119/30003},
   DOI={10.1209/0295-5075/119/30003},
   number={3},
   journal={EPL (Europhysics Letters)},
   publisher={IOP Publishing},
   author={Gurau, Razvan},
   year={2017},
   month=aug, pages={30003} }

@article{pens,
    author = "Penson, K. A. and Zyczkowski, K.",
    title = "Product of {G}inibre matrices: {F}uss-{C}atalan and {R}aney distributions",
    journal = "Physical Review",
    volume = "83",
    number = "6",
    pages = "061118",
    year = "2011",
}

@article {collins,
    AUTHOR = {Banica, T. and Belinschi, S. T. and Capitaine, M. and Collins,
              B.},
     TITLE = {Free {B}essel laws},
   JOURNAL = {Canad. J. Math.},
  FJOURNAL = {Canadian Journal of Mathematics. Journal Canadien de
              Math\'ematiques},
    VOLUME = {63},
      YEAR = {2011},
    NUMBER = {1},
     PAGES = {3--37},
      ISSN = {0008-414X,1496-4279},
   MRCLASS = {46L54 (15B52 60B20)},
  MRNUMBER = {2779129},
MRREVIEWER = {Nizar\ Demni},
       DOI = {10.4153/CJM-2010-060-6},
       URL = {https://doi.org/10.4153/CJM-2010-060-6},
}

@article {colgur2,
    AUTHOR = {Collins, Beno\^it and Gurau, Razvan and Lionni, Luca},
     TITLE = {The tensor {H}arish-{C}handra-{I}tzykson-{Z}uber integral
              {II}: detecting entanglement in large quantum systems},
   JOURNAL = {Comm. Math. Phys.},
  FJOURNAL = {Communications in Mathematical Physics},
    VOLUME = {401},
      YEAR = {2023},
    NUMBER = {1},
     PAGES = {669--716},
      ISSN = {0010-3616,1432-0916},
   MRCLASS = {81P15 (05C15 81P16 81P40 81P42)},
  MRNUMBER = {4604905},
MRREVIEWER = {Zhu-Jun\ Zheng},
       DOI = {10.1007/s00220-023-04653-5},
       URL = {https://doi.org/10.1007/s00220-023-04653-5},
}

@article {colgur1,
    AUTHOR = {Collins, Beno\^it and Gurau, Razvan and Lionni, Luca},
     TITLE = {The tensor {H}arish-{C}handra-{I}tzykson-{Z}uber integral {I}:
              {W}eingarten calculus and a generalization of monotone
              {H}urwitz numbers},
   JOURNAL = {J. Eur. Math. Soc. (JEMS)},
  FJOURNAL = {Journal of the European Mathematical Society (JEMS)},
    VOLUME = {26},
      YEAR = {2024},
    NUMBER = {5},
     PAGES = {1851--1897},
      ISSN = {1435-9855,1435-9863},
   MRCLASS = {05A05 (05A17 05C10 15B52)},
  MRNUMBER = {4735823},
MRREVIEWER = {Radhakrishnan\ Nair},
       DOI = {10.4171/jems/1315},
       URL = {https://doi.org/10.4171/jems/1315},
}

@article {sasa1,
    AUTHOR = {Kawano, Taigen and Obster, Dennis and Sasakura, Naoki},
     TITLE = {Canonical tensor model through data analysis: dimensions,
              topologies, and geometries},
   JOURNAL = {Phys. Rev. D},
  FJOURNAL = {Physical Review D},
    VOLUME = {97},
      YEAR = {2018},
    NUMBER = {12},
     PAGES = {124061, 25},
      ISSN = {2470-0010,2470-0029},
   MRCLASS = {83C45},
  MRNUMBER = {3893083},
       DOI = {10.1103/physrevd.97.124061},
       URL = {https://doi.org/10.1103/physrevd.97.124061},
}

@article {sasa2,
    AUTHOR = {Sasakura, Naoki and Sato, Yuki},
     TITLE = {Constraint algebra of general relativity from a formal
              continuum limit of canonical tensor model},
   JOURNAL = {J. High Energy Phys.},
  FJOURNAL = {Journal of High Energy Physics},
      YEAR = {2015},
    NUMBER = {10},
     PAGES = {109, front matter+18},
      ISSN = {1126-6708,1029-8479},
   MRCLASS = {83C45 (83C65)},
  MRNUMBER = {3435530},
MRREVIEWER = {Roman\ Romanovich\ Zapatrin},
       DOI = {10.1007/JHEP10(2015)109},
       URL = {https://doi.org/10.1007/JHEP10(2015)109},
}

@inproceedings {sedd0, 
    author = {Seddik, Mohamed El Amine and Tiomoko, Malik and Decurninge, Alexis and Panov, Maxim and Guillaud, Maxime}, title = {Learning from low rank tensor data: a random tensor theory perspective}, year = {2023}, publisher = {JMLR.org}, 
    booktitle = {Proceedings of the Thirty-Ninth Conference on Uncertainty in Artificial Intelligence}, 
    articleno = {174}, 
    numpages = {10}, 
    series = {UAI '23} 
}

@incollection {kun24,
    AUTHOR = {Kunisky, Dmitriy and Moore, Cristopher and Wein, Alexander S.},
     TITLE = {Tensor cumulants for statistical inference on invariant
              distributions},
 BOOKTITLE = {2024 {IEEE} 65th {A}nnual {S}ymposium on {F}oundations of
              {C}omputer {S}cience---{FOCS} 2024},
     PAGES = {1007--1026},
 PUBLISHER = {IEEE Computer Soc., Los Alamitos, CA},
      YEAR = {2024},
      ISBN = {979-8-3315-1674-1},
   MRCLASS = {68Q87},
  MRNUMBER = {4849263},
}

@article {au2023spectralasymptoticscontractedtensor,
    AUTHOR = {Au, Benson and Garza-Vargas, Jorge},
     TITLE = {Spectral asymptotics for contracted tensor ensembles},
   JOURNAL = {Electron. J. Probab.},
  FJOURNAL = {Electronic Journal of Probability},
    VOLUME = {28},
      YEAR = {2023},
     PAGES = {Paper No. 113, 32},
      ISSN = {1083-6489},
   MRCLASS = {60B20 (15B52 46L53 46L54)},
  MRNUMBER = {4650897},
MRREVIEWER = {Nam-Gyu\ Kang},
       DOI = {10.1214/23-ejp1001},
       URL = {https://doi.org/10.1214/23-ejp1001},
}

@misc{bonnin2025characterizationgaussiantensorensembles,
      title={Characterization of {G}aussian {T}ensor {E}nsembles}, 
      author={Remi Bonnin},
      year={2025},
      eprint={2505.02814},
      archivePrefix={arXiv},
      primaryClass={math-ph},
      url={https://arxiv.org/abs/2505.02814}, 
}

@article {male0,
    AUTHOR = {Male, Camille},
     TITLE = {Traffic distributions and independence: permutation invariant
              random matrices and the three notions of independence},
   JOURNAL = {Mem. Amer. Math. Soc.},
  FJOURNAL = {Memoirs of the American Mathematical Society},
    VOLUME = {267},
      YEAR = {2020},
    NUMBER = {1300},
     PAGES = {v+88},
      ISSN = {0065-9266,1947-6221},
%      ISBN = {978-1-4704-4298-9; 978-1-4704-6399-1},
   MRCLASS = {15B52 (18M60 46L54 60B20 60F05)},
  MRNUMBER = {4197072},
MRREVIEWER = {Nam-Gyu\ Kang},
       DOI = {10.1090/memo/1300},
       URL = {https://doi.org/10.1090/memo/1300},
}

@article {cameron,
    AUTHOR = {Cameron, Naiomi T. and McLeod, Jillian E.},
     TITLE = {Returns and hills on generalized {D}yck paths},
   JOURNAL = {J. Integer Seq.},
  FJOURNAL = {Journal of Integer Sequences},
    VOLUME = {19},
      YEAR = {2016},
    NUMBER = {6},
     PAGES = {Article 16.6.1, 28},
      ISSN = {1530-7638},
   MRCLASS = {05A15 (05A16 11B83)},
  MRNUMBER = {3546615},
MRREVIEWER = {Martin\ Klazar},
       DOI = {10.9734/bjmcs/2016/30398},
       URL = {https://doi.org/10.9734/bjmcs/2016/30398},
}

@article {forresterwang,
    AUTHOR = {Forrester, Peter J. and Wang, Dong},
     TITLE = {Muttalib-{B}orodin ensembles in random matrix
              theory---realisations and correlation functions},
   JOURNAL = {Electron. J. Probab.},
  FJOURNAL = {Electronic Journal of Probability},
    VOLUME = {22},
      YEAR = {2017},
     PAGES = {Paper No. 54, 43},
      ISSN = {1083-6489},
   MRCLASS = {60B20 (15B52)},
  MRNUMBER = {3666017},
MRREVIEWER = {Khanh\ Duy\ Trinh},
       DOI = {10.1214/17-EJP62},
       URL = {https://doi.org/10.1214/17-EJP62},
}

@article {chandra,
    AUTHOR = {Chandra, Tapas Kumar},
     TITLE = {de {L}a {V}all\'ee {P}oussin's theorem, uniform integrability,
              tightness and moments},
   JOURNAL = {Statist. Probab. Lett.},
  FJOURNAL = {Statistics \& Probability Letters},
    VOLUME = {107},
      YEAR = {2015},
     PAGES = {136--141},
      ISSN = {0167-7152,1879-2103},
   MRCLASS = {60A10 (60F05)},
  MRNUMBER = {3412766},
       DOI = {10.1016/j.spl.2015.08.011},
       URL = {https://doi.org/10.1016/j.spl.2015.08.011},
}

\bigskip
\bigskip
\noindent \textbf{Rémi Bonnin} \label{name} \\ 
{\em Ecole Normale Supérieure \\ 
Email: \href{mailto:remi.bonnin@ens.psl.eu}{remi.bonnin@ens.psl.eu} 

\end{document}